\documentclass[a4paper,10pt]{article}
\usepackage{amsmath}
\usepackage{amsfonts}
\usepackage{amsthm}

\usepackage[left=2.7cm,right=2.7cm,top=3.5cm,bottom=3cm]{geometry}
\usepackage{amsmath,amssymb,amsthm,amscd,amsfonts}
\usepackage{latexsym}
\usepackage{graphicx}
\usepackage[all]{xy}
\usepackage{color}
\usepackage[OT2,T1]{fontenc}
\usepackage{hyperref}
\usepackage{mathtools}
\usepackage{tikz}
\usetikzlibrary{arrows}
\usepackage{float}
\usepackage{caption}
\theoremstyle{plain}
\newtheorem{theorem}{Theorem}[section]

\newtheorem{proposition}[theorem]{Proposition}

\newtheorem{Problem}[theorem]{Problem}

\theoremstyle{definition}

\theoremstyle{remark}

\DeclareMathAlphabet{\mathpzc}{OT1}{pzc}{m}{it}
\DeclareSymbolFont{cyrletters}{OT2}{wncyr}{m}{n}
\DeclareMathSymbol{\Sha}{\mathalpha}{cyrletters}{"58}

\begin{document}
\thispagestyle{empty}

\newcommand\ddfrac[2]{\dfrac{\displaystyle #1}{\displaystyle #2}}

\newcommand{\ZZ}{{\mathbb Z}}
\newcommand{\Z}{{\mathbb Z}}
\newcommand{\Q}{\mathbb Q}
\newcommand{\FF}{{\mathbb F}}
\newcommand{\E}{{\mathcal{E}}}
\newcommand{\QQ}{\mathbb Q}
\newcommand{\mcF}{{\mathcal{F}}}
\newcommand{\A}{\textrm{A}}
\newcommand{\G}{\mathcal{G}}
\newcommand{\D}{\rm D}
\newcommand{\loc}{\textrm{loc}}
\newcommand{\chara}{\rm char}

\newcommand{\ds}{\displaystyle}
\newcommand{\la}{\langle}
\newcommand{\ra}{\rangle}
\newcommand{\z}{{\zeta}}
\newcommand{\ov}{\overline}
\newcommand{\wt}{\widetilde}
\newcommand{\Or}{\mathcal{O}}
\newcommand{\X}{\mathcal{X}}
\newcommand{\Hil}{\mathcal{H}}
\newcommand{\Ab}{\mathcal{A}}
\newcommand{\End}{\textrm{End}}
\newcommand{\Bi}{\mathfrak{B}}

\newcommand{\Gal}{{\rm Gal}}
\newcommand{\Id}{\rm Id}
\newcommand{\GL}{{\rm GL}}
\newcommand{\SL}{{\rm SL}}
\newcommand{\Dic}{{\rm Dic}}

\newcommand{\s}{\sigma}
\renewcommand{\a}{\alpha}
\renewcommand{\b}{\beta}
\renewcommand{\c}{\gamma}
\renewcommand{\d}{\delta}
\renewcommand{\leq}{\leqslant}
\renewcommand{\geq}{\geqslant}
\newcommand{\pf}{\textbf{Proof.}}
\newcommand{\eop}{\begin{flushright}$\square$\end{flushright}}
\newcommand{\modn}{{\textrm{mod} \hspace{0.1cm} }}

\title{On $7$-division fields of CM elliptic curves}
\author{Jessica Alessandr\`i and Laura Paladino\thanks{Corresponding author}}
\date{}

\maketitle

\renewcommand{\thefootnote}{\arabic{footnote}}
\setcounter{footnote}{0}
%----- first page ---------------------

%----- first page ---------------- Abstract
\begin{abstract}
	Let $\E$ be a CM elliptic curve defined over a number field $K$,
	with Weiestrass form $y^3=x^3+bx$ or $y^2=x^3+c$.
	For every positive integer $m$,
	we denote by $\E[m]$ the $m$-torsion subgroup of $\E$ and
	by $K_m:=K(\E[m])$ the $m$-th division field, i.e.\ the extension of $K$ generated by the coordinates
	of the points in $\E[m]$.  We classify all fields $K_7$. In particular we give explicit generators for $K_7/K$ and
produce all Galois groups $\Gal(K_7/K)$. We also show some applications to the Local-Global Divisibility
	Problem and to modular curves. 
\end{abstract}

\medskip
%----- first page ------------- Keywords
\medskip

%----- first page -------------- AMS classification codes
\textbf{Mathematics~Subject~Classification~(2010)}: 11G05, 11F80.

\section{Introduction} \label{sec1}
Let $K$ be a number field with algebraic closure $\ov{K}$ and let $\E$ be an elliptic curve defined over $K$. We keep the standard notation $\E[m]$ for the $m$-torsion subgroup of $\E$ and by $K_m$ we denote the $m$-th division field $K(\E[m])$, i.e.\ the field obtained by adding to $K$ the coordinates of the points in $\E[m]$.  Since the beginning of the studies on elliptic curves, the $m$-th division fields have played a key r\^ole. 
The properties of $K_m/K$ are related to Galois representations on the total Tate module, to Iwasawa theory, to modularity and  to the proof of the Mordell-Weil theorem.
The extension $K_m/K$ is a Galois extension, in fact it is the splitting field of the $m$-th division polynomial, i.e.\ the
polynomial whose roots are the abscissas of the $m$-torsion points of $\E$, and the polynomials whose
roots are the ordinates corresponding to those abscissas.  The extension $K_m/K$ is monogeneous by Artin's primitive element theorem, however, in general it is not easy to find an explicit single generator. It is also well-known that
$\E[m]\simeq (\Z/m\Z)^2$. Therefore, if $\{P_1\,,P_2\}$ is a generating set for $\E[m]$, with $P_i=(x_i,y_i)$, for $i=1,2$, then $K_m=K(x_1,x_2,y_1,y_2)$ and $\{x_1,x_2,y_1,y_2\}$ is the generating set for $K(\E[m])$ that is usually adopted.
We are interested in showing explicit generators for this extensions, searching for generating
sets as easy as possible to be used in applications. Indeed there are many potential applications, for instance in
Galois representations (see for example \cite{Sha}),  local-global problems on elliptic curves (see Subsection \ref{loc-glob}),
descent problems (see for example \cite{SS} and \cite{Ba} among others),
points on modular curves and points on Shimura curves. 
\par In the previous papers of this series \cite{BP},  \cite{BP2}, \cite{Pal_2018}, some of those applications have been showed,
as well as some new generating sets involving a primitive $m$-th root of the unity $\z_m$.
In fact, by the properties of the Weil pairing $e_m$, the image $\z_m:=e_m(P_1,P_2)\in K_m$ is a primitive
$m$-th root of unity  and $K(\z_m)\subseteq K_m$ (see for instance \cite{Sil}). It turned out that
$\z_m$ can be used as a generator for $K_m/K$ and in particular, when $m$ is odd, we have 
$ K_m=K(x_1,\zeta_m,y_2)$ \cite[Theorem 1.1]{BP2}.
\noindent When $m=p$ is a prime number, this generating set 
is minimal among the subsets of $\{x_1,x_2,\z_m,y_1,y_2 \}$ (for further details
see \cite{BP2}). On the contrary, in the case when $m=p^n$, with $n\geq 2$,
we can replace $\z_{p^n}$ with
$\z_p$, i.e.  $K_{p^n}=K(x_1,\zeta_p,y_2)$, for every $n\geq 1$ (see \cite[Theorem 1.1]{DP2}).
\par Observe that
when $K=\QQ$, we have $\z_m\notin K$ and therefore $\QQ(\E[m])\neq \QQ$, for every $m\geq 3$. 
In particular the extension $\QQ(\E[m])/\QQ$ is "as minimal as possible" when $\QQ(\E[m])=\QQ(\z_m)$.
L. Merel and M. Rebolledo proved that if such an equality holds when $m=p$ is a prime, then $p\leq 5$ (see \cite{Mer} and \cite{Reb}).
A classification of all elliptic curves such that  $\QQ(\E[3])=\QQ(\z_3)$ is given in \cite{Pal2} and a classification
of all elliptic curves such that  $\QQ(\E[5])=\QQ(\z_5)$ is given in \cite{GJ}. In this last paper E. Gonz\'alez Jim\'enez and
\`A. Lozano-Robledo also investigate the cases when 
$\QQ(\E[m])/\QQ$ is an abelian extension for all elliptic curves over number fields.
Among other important results, in particular they prove that if $\E$ is a CM elliptic curve and 
$\QQ(\E[m])/\QQ$ is abelian, then $m\in \{2,3,4,5,6,8\}$. 
\par For $m=3$ and $m=4$ there are explicit descriptions of all
possible fields $K_3$ and $K_4$, in terms of generators, degrees and
Galois groups \cite{BP2}, \cite{BP}. In \cite{Pal_2018} there is a similar description of the fields $K_5$ for the families of CM elliptic curves
$ \mcF_1 : y^2=x^3+bx $, with $b\in K$ and  
$\mcF_2 : y^2=x^3+c$, with  $c\in K $.
\par Here we give a classification of every possible field $K_7$, for the curves of the same families $ \mcF_1$
and $ \mcF_2$,
showing in particular explicit generators for the extension $K_7/K$. We also show all possible Galois groups $\Gal(K_7/K)$
for the curves in $ \mcF_1$
and $ \mcF_2$.
\par The paper is structured as follows. In the first part of it we describe generators, possible degrees and possible Galois groups for the curves
of the family $\mcF_1$. Then we give a similar description for the curves of the family $\mcF_2$. In the last part of the paper
we show some applications of these results. In particular we produce an application to the Local-Global Divisibility Problem, which was stated in \cite{DZ1} by R. Dvornicich and U. Zannier in 2001
(see Subsection \ref{loc-glob} for further details).  In addition we deduce some properties concerning points on modular curves.

\section{Generators of $K(\E[7])$ for elliptic curves $y^2=x^3+bx$}  \label{sec_gen_1}
For every positive integer $m$, the $m$-th division polynomial of an elliptic curve $\E$ is the polynomial whose roots are
the abscissas of the $m$-torsion points of $\E$. It is generally denoted by 
$\Psi_m(x)$. The polynomial
$\Psi_m$ has degree $\dfrac{m^2-1}{2}$ when $m$ is odd and $\dfrac{m^2-4}{2}$ when $m$ is even.
Let ${\mathcal{E}}_1$ be an elliptic curve defined over $K$, with Weierstrass form $y^2=x^3+bx$.
We will denote by $\phi_1$ the complex multiplication of $\E_1$, i.e. $\phi_1((x,y))=(-x,iy)$, for every point $P=(x,y)\in \E_1$.
Since $\phi_1$ is an automorphism of $\E_1$, then $\Psi_m$ is a polynomial
in $x^2$. When $m=7$, the $7$-th division polynomial of ${\mathcal{E}}_1$ is the polynomial

\[
\begin{split} q_7(x):=&7x^{24}+308bx^{22}-2954b^2x^{20}-19852b^3x^{18}-35231b^4x^{16}-82264b^{5}x^{14}-111916b^6x^{12}\\
	&-42168b^7x^{10}+15673b^8x^8+14756b^9x^6+1302b^{10}x^4+196b^{11}x^2-b^{12}.\\
\end{split}
\]

\bigskip\noindent We can set $t:=x^2$ and consider the polynomial $q_7(t)$ of degree 12
to look for the abscissas of the $7$-torsion points of $\E_1$. For every $\alpha\in K_7$, we denote by
$\overline{\alpha}$ its complex conjugate.
Let $i$ be a root of $x^2+1=0$, let $\sigma_1$ be the automorphism of the extension $\QQ(\z_7,i)/\QQ$ mapping $\z_7$ to
$\z_7^5$ and let

\[
\begin{split}
&\omega_1:=(6i + 4)\z_7^5  + (6i - 2)\z_7^4  + (2i - 2)\z_7^3  + (2i + 4)\z_7^2  + 8i \z_7 + 4i- 3;\\
& \\
%\omega_2&:= \overline{\omega_{1}}\\%=(- 6i + 4)\z_7^5  + (- 6i - 2)\z_7^4  + (- 2i - 2)\z_7^3  + (- 2i + 4)\z_7^2- 8i \z_7 - 4i- 3;\\
%& \\
&\omega_3:=\sigma_1(\omega_1)=(2i + 2)\z_7^5  + 6\z_7^4  + (- 4i + 6)\z_7^3  + (- 6i + 2)\z_7^2  - 4i \z_7 - 2i - 1;\\
%& \\
%\omega_4&:= \overline{\omega_{3}}\\%=(-2i + 2)\z_7^5  + 6\z_7^4  + (4i + 6)\z_7^3  + (6i + 2)\z_7^2  +4i \z_7 + 2i - 1;\\
& \\
&\omega_5:=\sigma_1(\omega_3)=(4i - 6)\z_7^5  + (- 2i - 4)\z_7^4  + (6i - 4)\z_7^3  - 6\z_7^2  + 4i \z_7 + 2i - 7;\\
& \\
%\omega_6&:= \overline{\omega_{5}}\\%=(-4i - 6)\z_7^5  + ( 2i - 4)\z_7^4  + (-6i - 4)\z_7^3  - 6\z_7^2  - 4i \z_7 - 2i - 7;\\
& \omega_{s+1}:= \overline{\omega_{s}}, \hspace{0.3cm} \textrm{for} \hspace{0.1cm} s\in \{1,3,5\};\\
\end{split}
\]
\bigskip
\[
\begin{array}{ll}
&\theta_1 :=\dfrac{1}{7}\big(( - 3520i - 1568)\z_7^5  + (- 4800i + 2352)\z_7^4+(- 256i + 2352)\z_7^3 \\
& \hspace{0.9cm}+ (- 1536i - 1568)\z_7^2  - 5056i \z_7-2528i+3584\big);\\
&\\
%&\theta_2:= \overline{\theta_{1}}\\%= \dfrac{1}{7}\big((  3520i - 1568)\z_7^5  + (4800i + 2352)\z_7^4+( 256i + 2352)\z_7^3\\
%&  \hspace{1.75cm}+ ( 1536i - 1568)\z_7^2  + 5056i \z_7+2528i+3584\big);\\
&\theta_3 := \sigma_1(\theta_1)=\dfrac{1}{7}\big(( - 256i - 2352)\z_7^5  + (1280i - 3920)\z_7^4+(3264i - 3920)\z_7^3\\
& \hspace{2.2cm} + (4800i - 2352)\z_7^2  + 4544i\z_7+2272i+1232\big);\\
&\\
%&\theta_4 :=  \overline{\theta_{3}}\\%= \dfrac{1}{7}\big(( 256i - 2352)\z_7^5  + (-1280i - 3920)\z_7^4+(-3264i - 3920)\z_7^3\\
%&  \hspace{1.75cm}  + (-4800i - 2352)\z_7^2  - 4544i\z_7-2272i+1232\big);\\
&\theta_5 :=\sigma_1(\theta_3)= \dfrac{1}{7}\big(( - 3264i + 3920)\z_7^5+ (1536i + 1568)\z_7^4+(- 3520i + 1568)\z_7^3\\
&  \hspace{2.2cm} + (1280i + 3920)\z_7^2- 1984i\z_7-992i+5152\big);\\
&\\
%&\theta_6 :=  \overline{\theta_{5}}\\%= \dfrac{1}{7}\big(( 3264i + 3920)\z_7^5+ (-1536i + 1568)\z_7^4+( 3520i + 1568)\z_7^3 \\
%&  \hspace{1.75cm} + (-1280i + 3920)\z_7^2+ 1984i\z_7+992i+5152\big).\\
& \theta_{s+1}:= \overline{\theta_{s}}, \hspace{0.3cm} \textrm{for} \hspace{0.1cm}  s\in \{1,3,5\}.\\
\end{array}
\]

%\[
%\begin{split}
%	\omega_1&:=(6i + 4)\z_7^5  + (6i - 2)\z_7^4  + (2i - 2)\z_7^3  + (2i + 4)\z_7^2  + 8i \z_7 + 4i- 3;\\
%	& \\
%	\omega_3&:=\sigma(\omega_1)=(2i + 2)\z_7^5  + 6\z_7^4  + (- 4i + 6)\z_7^3  + (- 6i + 2)\z_7^2  - 4i \z_7 - 2i - 1;\\
%	& \\
%	\omega_5&:=\sigma(\omega_3)=(-4i - 6)\z_7^5  + (2i - 4)\z_7^4  + (-6i - 4)\z_7^3  - 6\z_7^2  -4i \z_7 - 2i - 7.\\
%\end{split}
%\]

%\noindent For $1\leq h\leq 3$, we denote by $\omega_{2h}$ the complex conjugate $\overline{\omega_{2h-1}}$.
%In addition let

%\[
%\begin{array}{lll}
%&	\theta_1 := & \dfrac{1}{7}\big((- 3520i - 1568)\z_7^5  + (- 4800i + 2352)\z_7^4+(- 256i + 2352)\z_7^3 \\
%&	&  + (- 1536i - 1568)\z_7^2  - 5056i \z_7-2528i+3584\big);\\
%	&\theta_3  := &\sigma(\theta_1)=  \dfrac{1}{7}\big(( - 256i - 2352)\z_7^5  + (1280i - 3920)\z_7^4+(3264i - 3920)\z_7^3\\
%&	& + (4800i - 2352)\z_7^2  + 4544i\z_7+2272i+1232\big);\\
%&	\theta_5 := &\sigma(\theta_3)=  \dfrac{1}{7}\big(( 3264i + 3920)\z_7^5+ (-1536i + 1568)\z_7^4+( 3520i + 1568)\z_7^3 \\
%&	& + (-1280i + 3920)\z_7^2+ 1984i\z_7+992i+5152\big);\\
%\end{array}
%\]

\bigskip\noindent %and the corresponding complex conjugates $\theta_{2h}:=\overline{\theta_{2h-1}}$, for $1\leq h\leq 3$.
With the use of a software of computational algebra
(we used AXIOM, that is also implemented in SAGE), one can verify that $q_7(t)$ factors over $K(i,\z_7)$ as follows:

$$ q_7(t):= 7\prod_{j=1}^6   \left(t-\left(\omega_jb+\dfrac{1}{2}b\sqrt{\theta_j}\right) \right) \left(t-\left(\omega_jb-\dfrac{1}{2}b\sqrt{\theta_j}\right) \right). $$

\noindent Thus the roots of $q_7(x)$, i.e. the abscissas of the $7$-torsion points of $\E_1$, are

$$ x_{2j-1}=\sqrt{\left(\omega_jb+\dfrac{1}{2}b\sqrt{\theta_j}\right)}; \hspace{0.3cm} \hspace{0.3cm}  x_{2j}=\sqrt{\left(\omega_jb-\dfrac{1}{2}b\sqrt{\theta_j}\right)};$$

$$-x_{2j-1}=-\sqrt{\left(\omega_jb+\dfrac{1}{2}b\sqrt{\theta_j}\right)}; \hspace{0.3cm} \hspace{0.3cm}- x_{2j}=-\sqrt{\left(\omega_jb-\dfrac{1}{2}b\sqrt{\theta_j}\right)};$$

\noindent for $1\leq j\leq 6$. By using the equation $y^2=x^3+bx$, we can calculate the corresponding ordinates. 
For ease of notation, we will denote by $iP$ the point $\phi_1(P)=(-x,iy)$, where $P=(x,y)\in \E_1$.
It turns out that the 48 points of exact order 7 of $\E_1$ are the following:

\[
\begin{split}
	\pm P_{2j-1}&:=(x_{2j-1},\pm y_{2j-1})=\left(\sqrt{\omega_jb+\dfrac{1}{2}b\sqrt{\theta_j}}, \pm \sqrt{\left(\omega_j+\dfrac{1}{2}\sqrt{\theta_j}+1\right)b\sqrt{\omega_jb+\dfrac{1}{2}b\sqrt{\theta_j}}}\right); \\ 
\end{split} \]

\[
\begin{split}
	\pm P_{2j}&:=(x_{2j},\pm y_{2j})=\left(\sqrt{\omega_jb-\dfrac{1}{2}b\sqrt{\theta_j}}, \pm \sqrt{\left(\omega_j-\dfrac{1}{2}\sqrt{\theta_j}+1\right)b\sqrt{\omega_jb-\dfrac{1}{2}b\sqrt{\theta_j}}}\right); \\ 
\end{split} \]

$$\hspace{0.5cm} \pm iP_{2j-1}:=(-x_{2j-1},\pm iy_{2j-1}), \hspace{0.3cm} \textrm{ and } \hspace{0.3cm} \pm iP_{2j}:=(-x_{2j},\pm iy_{2j});$$

\noindent  for  $1\leq j\leq 6.$

\bigskip\begin{theorem} \label{gen2} Let  $\theta_j$ and $\omega_j$ be as above, for $j=1, ..., 6$ and let
	$\varepsilon\in \{+,-\}$ fixed. 
	Then \vskip 0.3cm \centerline{$K_7=K(i,\z_7,y_j)=K\left(i, \zeta_7,\sqrt{\left(\omega_j+\varepsilon\dfrac{1}{2}\sqrt{\theta_j}+1\right)b\sqrt{\omega_jb+\varepsilon\dfrac{1}{2}b\sqrt{\theta_j}}}\right).$}
\end{theorem}

\begin{proof}
	If $P$ is a
	nontrivial $7$-torsion point, then $\phi_1(P)$ is a $7$-torsion point too. If $\phi_1(P)$ is not a multiple of $P$,
	then a basis for $\E_1[7]$ is given by $\{P,\phi_1(P)\}$. Observe that $\phi_1(P)=nP$ if and only if $(i-n)P=O$. 
	Since the ring of automorphisms of $\E_1$ is $\mathbb{Z}[i]$ and $7$ is inert in $\mathbb{Z}[i]$, % (because of $7\equiv 3 \hspace{0.1cm} (\modn 4)$), 
then $\phi_1(P)$ is not a multiple of $P$, for every $P\in \E_1[7]$ of exact order 7, and we can choose $\{P_j,\phi_1(P_j)\}$ as a generating set of $\E_1[7]$, for any $j=1,\dots,12$. We have $K_7=K(x_j,y_j,-x_j,iy_j)=K(x_j,y_j,i)$.
	On the other hand, by \cite[Theorem 1.1]{BP2}, the field $K_7$ is equal to $K(x_j,\z_7, iy_j)$.
	Then in particular $K_7=K(x_j,i,\z_7,y_j)$. By calculating $y_j^2$ and $y_j^4$, one can verify that $\sqrt{\theta_j} \in K(i,\zeta_7,y_j^4)$.
	Thus $x_j\in K(i,\z_7,y_j)$ and we get the conclusion
	$$K_7=K(i,\z_7,y_j)=K\left(i, \zeta_7,\sqrt{\left(\omega_j+\varepsilon\dfrac{1}{2}\sqrt{\theta_j}+1\right)b\sqrt{\omega_jb+\varepsilon\dfrac{1}{2}b\sqrt{\theta_j}}}\right),$$

\noindent for every $j=1, ..., 6$ and $\varepsilon\in \{+,-\}$.
\end{proof}

\section{Degrees $[K_7:K]$ for the curves of $\mcF_1$} \label{sub3}

For ease of notation, from now on we will fix the generating set $\{P_1,\phi_1(P_1)\}$ for
$\E_1[7]$. By Theorem \ref{gen2} we have 
$K_7=K\left(i, \zeta_7,\sqrt{\left(\omega_1+\dfrac{1}{2}\sqrt{\theta_1}+1\right)b\sqrt{\omega_1b+\dfrac{1}{2}b\sqrt{\theta_1}}}\right)$.
As explained in the proof of Theorem \ref{gen2}, such a choice is without loss of generality and all the results
that we are going to show about the degree $[K_7:K]$ and the Galois group ${\Gal}(K_7/K)$ hold as
well for every other generating set of the extension $K_7/K$ listed in Theorem \ref{gen2}.

\begin{theorem} \label{deg1}
	Let $\E_1: y^2=x^3+bx$, with $b\in K$. 
	Let $y_1=\sqrt{\left(\omega_1+\dfrac{1}{2}\sqrt{\theta_1}+1\right)b\sqrt{\omega_1b+\dfrac{1}{2}b\sqrt{\theta_1}}}$ and consider the conditions
	
	\bigskip
	
	\begin{tabular}{llll}
		& {\bf \A.} \hspace{0.3cm} & $i\notin K$; \hspace{1.5cm}  & {\bf C.} \hspace{0.3cm} $\sqrt{\theta_1}\notin K(i,\z_7)$; \\
		& {\bf B1.} \hspace{0.3cm} & $\z_7+\z_7^{-1}\notin K(i)$;  & {\bf D.} \hspace{0.3cm} $\sqrt{\omega_1b+\dfrac{1}{2}b\sqrt{\theta_1}}\notin K(i,\z_7,\sqrt{\theta_1})$; \\
		& {\bf B2.} \hspace{0.3cm} & $\z_7\notin K(i,\z_7+\z_7^{-1})$; & {\bf E.} \hspace{0.3cm} $y_1 \notin K\left(i, \zeta_7, \sqrt{\omega_1b+\dfrac{1}{2}b\sqrt{\theta_1}}\right)$. \\
	\end{tabular}
	
	\bigskip
	
	The possible degrees of the extension $K_7/K$ are the following
	
	\begin{center}
		\begin{tabular}{|c|c|c|c|}
			\hline
			
			$d$ & {\em holding conditions}  &  $d$ & {\em holding conditions}  \\
			\hline
			{\em 96} &  {\bf A}, {\bf B1},  {\bf B2}, {\bf C}, {\bf D}, {\bf E} &  {\em 8} & {\bf E} \textrm{and one among } {\bf A}, {\bf B2}, {\bf C}, {\bf D}
			\\
			\hline
			{\em 48} &  {\bf B1}, {\bf E}, \textrm{and three among } {\bf A}, {\bf B2}, {\bf C}, {\bf D} & {\em 6} &  
			\begin{minipage}[t]{6cm}\begin{center} {\bf B1} {and one among } {\bf A}, {\bf B2}, {\bf E} \end{center}\end{minipage}\\
			\hline
		\begin{minipage}[t]{0.4cm}	\begin{center}\vskip 0.003cm	{\em 32} \end{center}\end{minipage}&
			\begin{minipage}[t]{6cm}	\begin{center}\vskip 0.0001cm	{\bf A},  {\bf B2}, {\bf C}, {\bf D}, {\bf E} \end{center}\end{minipage}  &  \begin{minipage}[t]{0.4cm}	\begin{center}\vskip 0.003cm{\em 4} \end{center}\end{minipage} &
			\begin{minipage}[t]{6cm}\begin{center} {\bf A}, {\bf B2} \textrm{ or } \\
					{\bf E} \textrm{and one among } {\bf A}, {\bf B2}, {\bf C}, {\bf D} \end{center}\end{minipage}\\
			\hline
		\begin{minipage}[t]{0.4cm}	\begin{center}\vskip 0.4cm {\em 24}  \end{center}\end{minipage}& \begin{minipage}[t]{6cm}\begin{center}\vskip 0.003cm {\bf B1}, {\bf E}, {\bf C}, {\bf D}\\
		\textrm{ or } {\bf B1}, {\bf E},	\textrm{one between} {\bf A} and {\bf B2} \textrm{and one between} {\bf C} and {\bf D}
\end{center}\vskip 0.005cm
\end{minipage} & \begin{minipage}[t]{0.4cm}	\begin{center}\vskip 0.4cm{\em 3} \end{center}\end{minipage}& \begin{minipage}[t]{6cm}	\begin{center}\vskip 0.4cm{\bf B1} \end{center}\end{minipage}
			\\
			\hline
			{\em 16} & {\bf E} and
			\textrm{three among } {\bf A},  {\bf B2}, {\bf C}, {\bf D}  &  {\em 2} &
			\begin{minipage}[t]{6cm}\begin{center}  
					\textrm{one among } {\bf A}, {\bf B2}, {\bf E}  \end{center}\end{minipage}\\
			\hline
			{\em 12} & 
					{\bf B1}, {\bf E} \textrm{and one among } {\bf A}, {\bf B2}, {\bf C}, {\bf D}  &  {\em 1} &
			\textrm{ no conditions hold } \\
			\hline
			\multicolumn{4}{c}{  }\\
			
			\multicolumn{4}{c}{{\em Table 1}} \\
			
		\end{tabular}
	\end{center}
	
\end{theorem}

\begin{proof}
	Consider the tower of extensions
	
	\[
	\begin{split}  K & \subseteq K(i)\subseteq K(i,\zeta_7+\zeta_7^{-1}) \subseteq K(i,\zeta_7) \subseteq K(i,\zeta_7,\sqrt{\theta_1})\subseteq K\left(i,\zeta_7, \sqrt{\omega_1b+\dfrac{1}{2}b\sqrt{\theta_1}}\right)\subseteq K\left(i, \zeta_7,y_1\right).\\
	\end{split}
	\]
	
	\noindent The degree $d:=[K_7:K]$ is the product of the degrees of the intermediate extensions appearing in the tower.
	A priori, each extension gives a contribute to the degree less than or equal to 2, except for the extension $K(i)\subset K(i,\zeta_7+\zeta_7^{-1})$ which gives a contribution dividing 3. However, some of the cases do not occur. 
%Observe that $\theta_1$
	%is not a square in $\QQ(i,\z_7)$ and  $\omega_1+\dfrac{1}{2} \sqrt{\theta_1}$
	%is not a square in $\QQ(i,\z_7,\sqrt{\theta_1})$. 
If {\bf E}
	does not hold, then $y_1\in K\left(i,\z_7,\sqrt{\omega_1b+\dfrac{1}{2}b\sqrt{\theta_1}}\right)$, i.e. $y_1 = \alpha + \beta \sqrt{\omega_1b+\dfrac{1}{2}b\sqrt{\theta_1}}$, with $\alpha,\beta\in  K(i,\z_7,\sqrt{\theta_1})$. By
$y_1^2=x_1^3+bx_1$, this implies $\sqrt{\omega_1b+\dfrac{1}{2}b\sqrt{\theta_1}}\in K(i,\z_7,\sqrt{\theta_1})$ and then {\bf D} does not hold. If both {\bf E} and {\bf D} do not hold, i.e.\  $x_1=\alpha_x+\beta_x\sqrt{\theta_1}$, and
$y_1=\alpha_y+\beta_y\sqrt{\theta_1}$ with $\alpha_x,\beta_x,\alpha_y,\beta_y\in K(i,\z_7)$, then in a similar way $\sqrt{\theta_1} \in K(i,\z_7)$, meaning that {\bf C} does not hold either. Therefore condition {\bf E}
	may not hold (implying that conditions {\bf D} and {\bf C} do not hold as well)
	only when  $d\leq \dfrac{96}{8}=12$ and $d\neq 8$. 
In addition, with the use of the software of computational algebra AXIOM, we have verified that $\theta_1$ is not a square in $\QQ(i,\z_7)$. If conditions {\bf A}, {\bf B1} and {\bf B2}  hold,
then we have that $K$ is linearly disjoint from $\QQ(i,\z_7)$ over $\QQ$ and $\theta_1$ is not a square in $K(i,\z_7)$ as well.
Then condition {\bf C} must hold too and $[K_7:K]\geq 24$. For what we have discussed before,
this also implies that {\bf E} holds and then $d\geq 48$. Thus, if $d\leq 24$, then 
 {\bf A}, {\bf B1} and {\bf B2} cannot hold at
the same time.
The final computation that gives the degree and the corresponding conditions in Table 1 is straightforward.
\end{proof}

 Notice that $[K_7:K]\leq 96< 366=|\GL_2(\Z/7\Z)|$ and the
Galois representation

\[ \rho_{\E_1,7}:{\Gal}(\ov{K}/K) \rightarrow  \GL_2(\Z/7\Z) \]

\noindent is not surjective, in accordance with $\E_1$ having complex multiplication.

\section{Galois groups ${\Gal}(K_7/K)$ for the curves of $\mcF_1$} \label{gal1}
Let $\E_1$ be a curve of the family $\mcF_1$, let $G:={\Gal}(K_7/K)$ and let $d:=|G|$. 
%\color{red} We denote by the
%same symbol both an automorphism of $\Gal(K_7/K)$ and the corresponding automorphism of  $\E_1[7]$. \normalcolor
Let $Q_{16}$ be the generalized quaternion group of order 16.

\begin{theorem} \label{galF1}
Let $K$ be a field with $\chara(K)\neq 2,3$ and let $\E_1$ be an elliptic curve with Weierstrass form
$y^2=x^3+bx$, where $b\in K$. Then $\Gal(K_7/K)$ is isomorphic to a subgroup
of $Q_{16}\rtimes \Z/6\Z$. In particular, if $[K_7:K]=96$, then $\Gal(K_7/K)\simeq Q_{16}\rtimes \Z/6\Z$.
\end{theorem}

\begin{proof}
Assume that all the conditions in Theorem \ref{deg1} hold. Then $[K_7:K]=96$. The image of $\Gal(\overline{K}/K)$ via the Galois representation $\rho_{\E_1,7}$ is a subgroup of $\GL_2(\Z/7\Z)$ isomorphic to $G=\Gal(K_7/K)$. We denote by $G$ both $\Gal(K_7/K)$ and its image in $\GL_2(\Z/7\Z)$. As a consequence of the properties of the Weil pairing, the action of
$\Gal(K_7/K)$ on $\z_7$ is via determinant, i.e.
$\sigma(\z_7)=\z_7^{\det(\sigma)}$, where $\sigma$ denotes both an element of $G$ and
its image in $\GL_2(\Z/7\Z)$. Consider the tower of extensions

\begin{figure}[H]
	\centering
	\begin{tikzpicture}
		\node (a) at (14,0) {$K_7$};
		\node (b) at (14,-2) {$K(\z_7)$};
		\node (c) at (14,-4) {$K$};
		\draw (a) -- (b);
		\draw (b) -- (c);
		\node (d) at (12.7,-2) {$G$};
		\node (e) at (15,-1) {$H$};
		\node (f) at (15,-3) {$G/H$};
		\draw[bend right]  (a) edge (c);
		\draw[bend left]  (a) edge (b);
		\draw[bend left]  (b) edge (c);
	\end{tikzpicture}
	\caption[]{}
	\label{fig:tower}
\end{figure}

\noindent We denote by $H$ both $\Gal(K_7/K(\z_7))$ and its image in $\GL_2(\Z/7\Z)$. We have that  the Galois group $\Gal(K(\z_7)/K)\simeq G/H$ is isomorphic to $\ZZ/6\ZZ$. 
If $\sigma$ fixes $\z_7$, then $\det(\sigma)=1$ and $\sigma \in \SL_2(7)$. 
Therefore
$H$ is isomorphic to a subgroup of $\SL_2(7)$ of order 16. 
It is well-known that $|\SL_2(7)|=336=16\cdot 21$. 
Therefore the image of $H$ in $\GL_2(\Z/7\Z)$ is a $2$-Sylow subgroup of $\SL_2(7)$. By Sylow's Theorems, the $2$-Sylow subgroups
are all conjugate and in particular they are all isomorphic. So it suffices to determine the structure of a $2$-Sylow subgroup
of $\SL_2(7)$ to get $H$ up to isomorphism. We know that one of the automorphisms of $G$ is the complex multiplication $\phi_1$.
We consider again $\{P_1, iP_1\}$ as a generating set of $\E_1[7]$. We have 

$$P_1 \xmapsto{\phi_1} iP_1  \xmapsto{\phi_1} -P_{1}  \xmapsto{\phi_1} -iP_1  \xmapsto{\phi_1}  P_1.$$

\noindent Then the representation of $\phi_1$ in $\GL_2(\Z/7\Z)$ is

$$\phi_1=\left(\begin{array}{cc}  0 & -1 \\ 1 & 0 \end{array}\right)$$

\noindent and $\det(\phi_1)=1$. Therefore $\phi_1\in H \lesssim \SL_2(7)$. Observe that $\phi_1^2=-\Id$.

\noindent Consider the matrix

$$\tau_1=\left(\begin{array}{cc}  2 & 1 \\ 1 & 1 \end{array}\right).$$

\noindent Since $\det(\tau_1)=1$, then $\tau_1\in \SL_2(7)$. In addition $\tau_1$ has order 8 and in particular
$\tau_1^{4}=-\Id$. One can easily verify that $\phi_1\tau_1=\tau_1^{-1}\phi_1$. Therefore the group generated by $\phi_1$ and
$\tau_1$ has the following presentation

$$\langle \phi_1, \tau_1 | \phi_1^2=\tau_1^{4}=-\Id, \phi_1\tau_1=\tau_1^{-1}\phi_1\rangle$$

\noindent and it is then isomorphic to the generalized quaternion group $Q_{16}$, i.e.\ the dicyclic group $\textrm{Dic}_{4}$.
This is a group of order 16, hence it is a $2$-Sylow subgroup of $\SL_2(7)$. Thus $H$ is isomorphic to $Q_{16}$ too. 
We have 

$$H=\langle \phi_1, \varphi_1 | \phi_1^2=\varphi_1^{4}=-\Id, \phi_1\varphi_1=\varphi_1^{-1}\phi_1\rangle,$$

\noindent where $\varphi_1$ is a conjugate of  $\tau_1$. To deduce $G$, we have to look more closely at the automorphism generating $G/H\simeq \ZZ/6\ZZ$. In fact, $\GL_2(\Z/7\Z)\simeq \SL_2(7) \rtimes  \FF_7^*\simeq \SL_2(7) \rtimes \Z/6\Z$, where $\FF_7$ is the finite field with $7$ elements. Thus $G \simeq H \rtimes \Z/6\Z\simeq Q_{16}\rtimes \Z/6\Z$. We are going to show that this semidirect product is not a direct product.
 The group $\Gal(K(\z_7)/K)\simeq G/H$ is generated by an automorphism $\psi_1$ corresponding to the automorphism $\sigma_1$ of $\Q(\z_7,i)/\Q$ mapping $\z_7$ to $\z_7^5$. Since $\sigma_1(\omega_1)=\omega_3$ and $\sigma_1(\theta_1) = \theta_3$, then we have that $\psi_1$ acts on the basis $\{P_1,iP_1\}$ by mapping $P_1$ to one of the points
$\pm P_{s}$, $\pm iP_{s}$, for some $s\in \{5,6\}$. Observe that if $x_1$ is sent to $x_s$ (respectively $-x_s$), for some $s\in \{5,6\}$, then $-x_1$ is sent to $-x_s$ (resp. $x_s$). 
Therefore if $\psi_1$ maps the point $P_1$ to $P_s$ (resp. $-P_s$) then $\psi_1$ maps the point $iP_1$ to one of the points $\pm iP_{s}$. Similarly
if $\psi_1$ maps the point $P_1$ to  $iP_{s}$ (resp. $-iP_s$), then $\psi_1$ maps the point $iP_1$ to one of the points $\pm P_{s}$. Let $\psi_1(P_1)=\alpha P_1 +\beta iP_1$. Then $\psi_1(iP_1)=\pm (-\beta P_1 +\alpha iP_1)$, i.e.

$$\psi_1=\left(\begin{array}{cc}  \alpha & -\beta \\ \beta & \alpha \end{array}\right) \quad\text{or}\quad \psi_1=\left(\begin{array}{cc}  \alpha & \beta \\ \beta & -\alpha \end{array}\right). $$

\noindent Since $iP_1=\phi_1(P_1)$, if $\psi_1$ and $\phi_1$ commute, then $\psi_1(iP_1)=\phi_1(\psi_1(P_1))=\alpha iP_1 -\beta P_1$. In this case we have such a representation of
$\psi_1$:

$$\psi_1=\left(\begin{array}{cc}  \alpha & -\beta \\ \beta & \alpha \end{array}\right).$$

\noindent Observe that every power of $\psi_1$ is a matrix of the same type

$$\psi_1^n=\left(\begin{array}{cc}  \alpha_n & -\beta_n \\ \beta_n & \alpha_n \end{array}\right),$$

\noindent for some $\alpha_n, \beta_n\in \ZZ/7\ZZ$. In particular

$$\psi_1^3=\left(\begin{array}{cc}  \alpha^3-3\alpha\beta^2 & \beta^3-3\alpha^2\beta \\  -\beta^3+3\alpha^2\beta  &  \alpha^3-3\alpha\beta^2  \end{array}\right)=\left(\begin{array}{cc}  \alpha_3 & -\beta_3 \\ \beta_3 & \alpha_3 \end{array}\right),$$

\noindent Since $G/H\simeq \Z/6\Z$ and
$G= H \rtimes G/H$, then $\psi_1^6=\Id$. Hence

$$(\psi_1^3)^2=\left(\begin{array}{cc}  \alpha_3^2 -\beta_3^2 & -2\alpha_3\beta_3 \\ 2\alpha_3\beta_3 & \alpha_3^2 -\beta_3^2 \end{array}\right)\equiv \left(\begin{array}{cc} 1 & 0 \\ 0 & 1 \end{array}\right) (\modn 7).$$

\noindent Thus $\alpha_3\beta_3\equiv 0 \hspace{0.1cm} (\modn 7)$, implying $\alpha_3=0$ or $\beta_3=0$. Therefore
$\alpha^3-3\alpha\beta^2\equiv 0 \hspace{0.1cm} (\modn 7)$ or $\beta^3-3\alpha^2\beta\equiv 0 \hspace{0.1cm} (\modn 7)$, i.e.\ $\alpha=0$ or $\beta=0$ or $\alpha^2\equiv 3\beta^2 \hspace{0.1cm} (\modn 7)$
or $\beta^2\equiv 3\alpha^2 \hspace{0.1cm} (\modn 7)$. The last two congruences have no solutions in $\ZZ/7\ZZ$. Thus $\alpha=0$ or $\beta=0$. 
Assume $\alpha=0$, then

$$\psi_1=\left(\begin{array}{cc} 0 & -\beta \\ \beta & 0 \end{array}\right)= \beta \phi_1.$$

\noindent Since  $\psi_1^6=\Id$  and $\phi_1^2=-\Id$, then $-\beta^6\equiv 1 \hspace{0.1cm} (\modn 7)$ and we have a contradiction with Fermat's Little Theorem. If $\beta=0$, then the automorphism $\psi_1$ is represented by a scalar matrix $\alpha\cdot \Id$. Since $\psi_1$ acts on $\z_7$ via determinant and we are assuming that $\psi_1$ is the automorphism of order 6 induced by the automorphism $\sigma_1$ mapping $\z_7$ to $\z_7^5$, then $\det(\psi_1)\equiv 5 \hspace{0.1cm} (\modn 7)$, i.e. $\alpha^2\equiv 5 \hspace{0.1cm} (\modn 7)$. This congruence has no solutions in $\Z/7\Z$.
Therefore $$\psi_1=\left(\begin{array}{cc}  \alpha & \beta \\ \beta & -\alpha \end{array}\right)$$

\noindent and $\psi_1$ and $\phi_1$ do not commute. Hence $G$ is not isomorphic to $Q_{16}\times \Z/6\Z$. 
If $[K_7:K]< 96$, then $G$ is isomorphic to a proper subgroup of $Q_{16}\rtimes \ZZ/6\ZZ$.
\end{proof}

\bigskip We are going to describe the Galois group $G={\Gal}(K_7/K)$ (up to isomorphism), for all possible $d:=[K_7:K]\leq 96$. 
In the last part of the proof of Theorem \ref{galF1} we have showed that 

$$\psi_1=\left(\begin{array}{cc}  \alpha & \beta \\ \beta & -\alpha \end{array}\right).$$

\noindent Hence

$$\psi_1^2=\left(\begin{array}{cc}  \alpha^2+\beta^2 & 0 \\ 0 & \alpha^2+\beta^2 \end{array}\right)=-\det(\psi_1)\Id$$

\noindent and then $\psi_1^2$ and $\psi_1^4$ commute with every other automorphism of $G$.
Observe that instead $\psi_1^3=-\det(\psi_1)\psi_1$ does not commute with $\phi_1$.
Moreover observe that $\phi_1^2=\varphi_1^4=-\Id$ is an automorphism of $\Gal(K_7/K(i))$. By \cite[Chapter II, Theorem 2.3]{Sil2}, the extension  $K_7/K(i)$ is abelian. If all the conditions in the statement of Theorem \ref{gen2} hold, then  $\Gal(K_7/K(i))=\langle \varphi_1,\psi_1 \rangle\simeq \Z/8\Z \times \Z/6\Z$. In particular $\psi_1$ commutes with 
$\varphi_1$. We also have that if condition {\bf A} does not hold, then $G$ is abelian.
Furthermore we recall that every subgroup of a generalized quaternion group is cyclic or it is a (generalized) quaternion group itself. The only proper non abelian subgroup of $Q_{16}$ is $Q_8$. The other proper nontrivial subgroups of $Q_{16}$ are the groups $\ZZ/m\ZZ$, with $m\in \{2,4,8\}$.
 If $Q_8$ is a subgroup of $G$ (and thus of $H$) under certain conditions, clearly we have that $\phi_1\in H$. In fact, if  $\phi_1\notin H$, then $H$ is an abelian group. If {\bf E} and at least one between {\bf C} and {\bf D} hold, then $\phi_1\in H$. We also observe that  $\phi_1\varphi_1^n=\varphi_1^{-n}\phi_1$, for every power $\varphi_1^n$ of $\varphi_1$.

\bigskip \emph{Galois groups $\Gal(K(E_1[7])/K)$}

\begin{description}
	
	\item[$d=96$]
	If the degree $d$ of the extension $K_7/K$ is 96, then all the conditions in Table 1 hold. We have already proved in Theorem \ref{galF1} that $G \simeq Q_{16} \rtimes \ZZ/6\ZZ$.
	
	\item[$d=48$]
	If the degree $d$ of the extension $K_7/K$ is 48, then condition {\bf B1} holds, because of $3|d$. If {\bf E} does not hold, then
	as stated in the proof of Theorem \ref{deg1} we have that consequently {\bf D} and {\bf C} do not hold and
	we would have an extension of degree $d<48$. Therefore condition {\bf E} holds.

\begin{description}
	\item[-] if {\bf A} does not hold, then $G$ is abelian, as mentioned above. We have $G= \langle \varphi_1, \psi_1\rangle \simeq
\ZZ/8\ZZ\times \ZZ/6\ZZ$.
 \item[-] if {\bf A}  holds, then at least one among {\bf C}  and  {\bf D} holds. As stated above, condition {\bf E}
	holds as well. Thus $\phi_1$ is  an automorphism of order 4 of $G$ and $G$ is not abelian (recall that we have already observed that $\phi_1$ does not commute with any map in $Q_{16}$ except its powers). 
\begin{description} 

\item[] If {\bf B2} holds, then $\psi_1$ has order 6 and  $G \simeq Q_{8} \rtimes \ZZ/6\ZZ$. 

\item[] If {\bf B2} does not hold, then $G/H\simeq \ZZ/3\ZZ$ is generated by $\psi_1^2$, which is represented by a scalar matrix, as we have observed above. Thus $G \simeq Q_{16} \times \ZZ/3\ZZ$.
\end{description}
	\end{description}
	
	\item[$d=32$]
	If the degree $d$ of the extension $K_7/K$ is 32, then all the conditions hold but {\bf B1}. Thus we have that $G/H\simeq \ZZ/2\ZZ$
	is generated by the automorphism $\psi_1^3$ mapping $\z_7$ to $\z_7^{-1}$ and $G \simeq Q_{16} \rtimes \ZZ/2\ZZ$.
	
	\item[$d=24$]
	Condition {\bf B1} must hold in all cases when $d=24$, because of $3|d$. We also have that {\bf E} holds (recall that if  {\bf E} does not hold, then  both {\bf C} and {\bf D} do not hold as well and $d<24$).
	\begin{description}
	\item[-]  If {\bf A} does not hold, then we have a subgroup of $\ZZ/8\ZZ\times \ZZ/6\ZZ$.
	\begin{description}
	\item[] Therefore, if  {\bf B2} does not hold too, then  $G/H\simeq \ZZ/3\ZZ$
	and $G\simeq \ZZ/8\ZZ\times \ZZ/3\ZZ$. 
\item[] If {\bf B2} holds, then $G/H\simeq \ZZ/6\ZZ$ and
	$G\simeq \ZZ/4\ZZ\times \ZZ/6\ZZ$.
\end{description}
\item[-] 	 Assume that {\bf A} holds.  By the proof of Theorem \ref{deg1}, we have that {\bf B2} does not hold and	 one between {\bf C} and {\bf D} holds. Then $\phi_1\in G$ and $G/H = \langle \psi_1^2 \rangle \simeq \ZZ/3\ZZ$. Therefore $H$ is a subgroup of $Q_{16}$ of order 8, which is not abelian (recall that $\phi_1$ does not commute with any map in $Q_{16}$ except its powers). Thus $G\simeq Q_{8} \times  \ZZ/3\ZZ$.
\end{description}

	\item[$d = 16$]
	If the degree $d$ of the extension $K_7/K$ is 16, as stated in Table 1,  then condition {\bf B1} does not hold and {\bf E} holds. Only one among the other conditions does not hold. 
	\begin{description}
	\item[-] If {\bf B2} does not hold, then %$K(\z_7,i)=K(i)$ and 
	$G/H$ is trivial. Therefore $G\simeq Q_{16}$.
\item[-] If {\bf B2} holds, then $G/H\simeq \langle \psi_1^3\rangle\simeq \ZZ/2\ZZ$ and $|H|=8$. 
\begin{description}
	\item[]
If  {\bf A} does not hold, 
	then we have an abelian extension and
	$G\simeq \ZZ/8\ZZ\times \ZZ/2\ZZ$. 
	\item[] Assume that {\bf A} holds. Since one between
	{\bf C} and {\bf D} holds as well, then $\phi_1 \in G$ and $H$ is not abelian. 
	Therefore $G\simeq Q_8\rtimes \ZZ/2\ZZ$.
	\end{description}
	\end{description}

	\item[$d = 12$]
	If the degree $d$ of the extension $K_7/K$ is 12, then $Q_8$ cannot be a subgroup of $G$. Condition {\bf B1} holds because of $3|d$. We have the following cases.
\begin{description}
	\item[-] If {\bf B2} holds, then $G/H$  has order 6 and $H$ has order 2 and it is generated by $-\Id$. Thus $G$ is abelian and $G\simeq  \Z/3\Z\times (\Z/2\Z)^2$.
	\item[-] If {\bf B2} does not hold, then  $G/H$  has order 3 and is generated by $\psi_1^2$. In this case $H$ has order 4. Since every abelian subgroup of $Q_{16}$ is cyclic and $\psi_1^2$ commutes with every other automorphism of $G$, then $G\simeq \Z/4\Z\times \Z/3\Z$.
		\end{description}
	
	\item[$d=8$]
	If the degree $d$ of the extension $K_7/K$ is 8, then {\bf B1} does not hold and {\bf E} holds.
\begin{description}
	\item[-] If {\bf A} does not hold, then we have an abelian extension. 
\begin{description}
	\item[] If {\bf B2} holds, then $G/H$  has order 2, $H$ has order 4
and $G\simeq \Z/4\Z\times \Z/2\Z$. 
\item[] If {\bf B2} does not hold, then $G=H\simeq \Z/8\Z$. 
	\end{description}
	\item[-] Assume that {\bf A} holds. 
\begin{description}
	\item[]  If  {\bf B2} does not hold, then $G=H$ has order 8. In this case both  
 {\bf C} and {\bf D} hold, then $\phi_1$ is an automorphism of $G$ and
$G\simeq Q_8$. 
	\item[] If  {\bf B2} holds, then one between  {\bf C} and {\bf D} holds. 
We have that $G/H$  has order 2 and $H$ has order 4. The complex multiplication $\phi_1$ is
an automorphism of $G$ and then $G\simeq \Z/4\Z\rtimes \Z/2\Z$.
	\end{description}
	\end{description}
	
	\item[$d=6$]
	If the degree $d$ of the extension $K_7/K$ is 6, then  condition {\bf B1} must hold in all cases, as listed in Table 1, and $G/H$ has order divided by 3. Then in every case we have an abelian group of order 6, i.e. $G \simeq \ZZ/3\ZZ\times \ZZ/2\ZZ$.
	
	\item[$d=4$]
	If  the degree $d$ of the extension $K_7/K$ is 4, then {\bf B1} does not hold and two of the other conditions hold. 
\begin{description}
	\item[-] If {\bf B2} does not hold, then $G/H$ is trivial and $G=H$ is isomorphic to a subgroup of $Q_{16}$
	of order $4$. Thus $G\simeq \ZZ/4\ZZ$. 
\item[-] If  {\bf B2} holds, then $G/H\simeq \ZZ/2\ZZ$ and $G$ is isomorphic to the Klein group $\ZZ/2\ZZ\times\ZZ/2\ZZ$.
	\end{description}
	
	\item[$d\leq 3$]
	If the degree $d$ of the extension $K_7/K$ is 3, 2 or 1, the Galois group is respectively $\Z/3\Z$, $\Z/2\Z$ or $\{{\rm Id}\}$.
	
\end{description}

\normalcolor

\section{Generators of the $7$-th division field for elliptic curves $y^2=x^3+c$} \label{subgen1}
Let $\E_2$ be an elliptic curve with Weierstrass form $y^2=x^3+c$, with $c\in K$.
Then the $m$-th division polynomial $\Psi_m(x)$  of $\E_2$ is a polynomial in $x^3$, because
of the automorphism of $\E_2$ given by the complex multiplication $\phi_2$, which maps $\left(x,y\right)$ to $\left(\z_3x,y\right)$.
If $m=p$ is an odd prime, then $\Psi_p(x)$ has degree $\dfrac{p^2-1}{2}$. Observe that $3| p^2-1$. Set $t:=x^3$, then $\Psi_p(t)$ 
is a polynomial of degree $\dfrac{p^2-1}{6}$  in the variable $t$. For $1\leq i\leq \dfrac{p^2-1}{6}$, let $\delta_j$ 
be the roots of $\Psi_p(t)$.
Therefore
the $p^2-1$ abscissas of the $p$-torsion points of $\E_2$ of exact order $p$ are
$\left\{\sqrt[3]{\delta_jc}, \z_3\sqrt[3]{\delta_jc}, \z_3^2\sqrt[3]{\delta_jc}\hspace{0.1cm}  \big|\hspace{0.1cm} 1\leq j\leq (p^2-1)/6\right\}$.
We also have that the ordinates of the points with abscissas in  $\left\{\sqrt[3]{\delta_j}, \z_3\sqrt[3]{\delta_j}, \z_3^2\sqrt[3]{\delta_j}\right\}$ are $\pm \sqrt{(\delta_j+1)c}$. The point  $\phi_2\left(\left(\sqrt[3]{\delta_jc},\sqrt{\delta_j+c}\right)\right)=
\left(\zeta_3\sqrt[3]{\delta_jc},\sqrt{\delta_j+c}\right)$ is still a $p$-torsion point of $\E_2$, for every $\delta_j$.
If $P_j=\left(\sqrt[3]{\delta_jc},\sqrt{\delta_j+c}\right)$ and $\phi_2\left(P_j\right)$ are linearly independent, then
$\{P_j,\phi_2\left(P_j\right)\}$ is a generating set for $\E_2[p]$. In this case, we have both
$K\left(\E_2[p]\right)=K\left(\sqrt[3]{\delta_jc},\z_3,\sqrt{\left(\delta_j+1\right)c}\right)$ and $K\left(\E_2[p]\right)=K\left(\sqrt[3]{\delta_jc},\z_p,\sqrt{\left(\delta_j+1\right)c}\right)$
(this last equality following by \cite[Theorem 1.1]{BP2}).
We now make these generating sets explicit for $p = 7$, by producing the coordinates of the points in $\E_2[7]$.

\bigskip\par We denote by $r_7(x)$ the $7$-th division polynomial of a curve $\E_2\in \mcF_2$. We have

$$
\begin{array}{ll}
	r_7(x):=& 7x^{24}+3944c \hspace{0.05cm}x^{21}-42896c^2x^{18}-829696c^3x^{15}-928256c^4x^{12}\\
	&-1555456c^5x^9-2809856c^6x^6-802816c^7x^3+65536c^8.\\
\end{array}$$

\noindent Let $\sigma_2$ be the automorphism of $\QQ(\z_3,\z_7)/\QQ$ mapping $\zeta_7$ to $\z_7^5$,
let $\varphi$ be the automorphism of $\QQ(\z_3,\z_7)/\QQ$ mapping  $\z_3$ to $\z_3^2$ and let

 \[
\begin{split} \delta_1 & :=-((- 132\zeta_3 - 120)\zeta_7^5  + (- 168\zeta_3 - 12)\zeta_7^4  + (- 24\zeta_3 + 60)\zeta_7^3\\
& \hspace{0.85cm}+(- 60\zeta_3 - 84)\zeta_7^2  + (- 192\zeta_3 - 96)\zeta_7 - 96\zeta_3 + 52); \\
 \delta_2 & :=\sigma_2(\delta_1)=-((- 24\zeta_3 - 84)\zeta_7^5  + (36\zeta_3 - 108)\zeta_7^4  + (108\zeta_3 - 72)\zeta_7^3\\
& \hspace{0.85cm} +(168\zeta_3 + 12)\zeta_7^2  + (144\zeta_3 + 72)\zeta_7 + 72\zeta_3 + 64); \\
 \delta_3 & :=\sigma_2(\delta_2)=-((108\zeta_3 + 180)\zeta_7^5  + (- 60\zeta_3 + 24)\zeta_7^4  + (132\zeta_3 + 120)\zeta_7^3\\
&\hspace{0.85cm}+(- 36\zeta_3 + 108)\zeta_7^2  + (72\zeta_3 + 36)\zeta_7 + 36\zeta_3 + 172); \\
 \delta_4 & :=\sigma_2(\delta_3)=-((132\zeta_3 + 12)\zeta_7^5  + (168\zeta_3 + 156)\zeta_7^4  + (24\zeta_3 + 84)\zeta_7^3 \\
& \hspace{0.85cm} + (60\zeta_3 - 24)\zeta_7^2+(192\zeta_3 + 96)\zeta_7 + 96\zeta_3 + 148); \\
\delta_5 & :=\sigma_2(\delta_4)=-((24\zeta_3 - 60)\zeta_7^5  + (- 36\zeta_3 - 144)\zeta_7^4  + (- 108\zeta_3 - 180)\zeta_7^3\\
& \hspace{0.85cm}+(- 168\zeta_3 - 156)\zeta_7^2  + (- 144\zeta_3 - 72)\zeta_7 - 72\zeta_3 - 8); \\
  \delta_6 & :=\sigma(\delta_5)=-((- 108\zeta_3 + 72)\zeta_7^5  + (60\zeta_3 + 84)\zeta_7^4  + (- 132\zeta_3 - 12)\zeta_7^3\\
&  \hspace{0.85cm}+(36\zeta_3 + 144)\zeta_7^2  + (- 72\zeta_3 - 36)\zeta_7 - 36\zeta_3 + 136); \\
 \delta_7  &:=\ddfrac{ 12\z_3 +8}{7}; \\
  \delta_8 & :=\varphi(\delta_8)=-\ddfrac{ 12\z_3 +4}{7}. \\
 \end{split}
 \]

\noindent The polynomial $r_7(x)$ factors over $K(\z_3,\zeta_7)$ as follows

$$	r_7(x)=  7 \prod_{j=1}^{8}( x^3 +\delta_jc).$$

\bigskip \noindent Then, as mentioned above, the 48 torsion points of $\E_2$ with exact order 7 are:

\[
\begin{array}{llllll}
	\pm P_j &=(x_j,\pm y_j) &=\left(\sqrt[3]{\delta_jc}, \pm  \sqrt{(\delta_j+1)c}\right); \\
	\pm \phi(P_j)  &=(\z_3x_j,\pm y_j) &=\left(\z_3\sqrt[3]{\delta_jc}\hspace{0.1cm}, \pm  \sqrt{(\delta_j+1)c}\right);  \\
	\pm \phi^2(P_j) &=(\z_3^2x_j,\pm y_j) &= \left(\z_3^2\sqrt[3]{\delta_jc}\hspace{0.1cm}, \pm  \sqrt{(\delta_j+1)c}\right); \\
\end{array} \]

\noindent for $1\leq j\leq 8$.

\begin{theorem} \label{zeta3} Let $\delta_j$ be as above, with $1\leq j\leq 6$.
	Then
	
	$$K_7=K\left(\sqrt[3]{\delta_jc},\zeta_3,\sqrt{(\delta_j+1)c}\right)=K\left(\sqrt[3]{\delta_jc},\zeta_7,\sqrt{(\delta_j+1)c}\right).$$
\end{theorem}

\begin{proof} 
	We have already observed at the beginning of this section that if $\phi_2(P)$ is a $7$-torsion point that is not a multiple of $P$,
	then a basis for $\E_2[7]$ is given by $\{P,\phi_2(P)\}$ and $K_7=K(x(P),\z_3,y(P))$.  However, in some cases the point $\phi_2(P)$ is a multiple of $P$. This happens for the points $P_j$ and $\phi_2(P_j)$, when $j=7,8$
	(in fact if $\z_7\notin K$, then $\z_7\notin K(x(P_j),\z_3,y(P_j))$, 
	for $j=7,8$, contradicting the well-known property of the Weil pairing recalled above). By the use of a software of computational algebra (we used AXIOM again), one can verify that $x(2P_1)=x(P_3)$ (i.e. $2P_1=P_3$ or $2P_1=-P_3$) and $x(4P_1)=x(P_5)$ (i.e. $4P_1=P_5$ or $4P_1=-P_5$).
Suppose that $\phi_2(P_1) = n P_1$. Since $\phi_2^2(P_1)=-P_1-\phi_2(P_1)$, we have $(n^2+n+1)P_1 = O$. Thus $n$ is a root of $n^2+n+1$ modulo $7$, hence $n \equiv 2,4 \hspace{0.1cm} (\modn 7)$. But, as noticed above, we have $x(2P_1) \neq x(\phi_2(P_1))$ and $x(4P_1) \neq x(\phi_2(P_1))$.
Thus $\phi_2(P_1)$ and $P_1$ are linearly independent.
Similar arguments apply for $P_3$ and $P_5$. In addition one can verify that
	$x(2P_2)=x(P_4)$ and $x(4P_2)=x(P_6)$ and repeat the arguments for those points too. Therefore
	$P_j$ and $\phi_2(P_j)$ are linearly
	independent for $j=1, ..., 6$ and
	$\{P_j,\phi_2(P_j)\}$ is a basis of $\E_2[7]$, for every $j=1, ..., 6$. Then $K_7=K\left(\sqrt[3]{\delta_jc},\zeta_3,\sqrt{(\delta_j+1)c}\right)$.
	As stated above, by \cite[Theorem 1.1]{BP2} we also have  $K_7=
	K\left(\sqrt[3]{\delta_jc},\zeta_7,\sqrt{(\delta_j+1)c}\right)$. 
\end{proof}

\section{Degrees $[K_7:K]$ for the curves of $\mcF_2$} \label{sub1}

By the results achieved in Theorem \ref{zeta3}, we are going to describe the possible degrees
$[K_7:K]$ for the elliptic curves of the family $\mcF_2$. From now on we will fix the generating set $\{P_1,\phi_2(P_1)\}$ for
$\E_2[7]$. Thus $K_7=K\left(\sqrt[3]{\delta_1c},\zeta_3,\sqrt{(\delta_1+1)c}\right)=K\left(\sqrt[3]{\delta_1c},\zeta_7,\sqrt{(\delta_1+1)c}\right)$.
Clearly all the results
that we are going to show about the degree $[K_7:K]$ and the Galois group ${\Gal}(K_7/K)$ hold as
well for every other generating set $\left\{\sqrt[3]{\delta_jc},\zeta_3,\sqrt{(\delta_j+1)c}\right\}$ or $\left\{\sqrt[3]{\delta_jc},\zeta_7,\sqrt{(\delta_j+1)c}\right\}$  of the extension $K_7/K$,
with $2\leq j\leq 6$.

\begin{theorem} \label{conditions2}
	Let $\E_2: y^2=x^3+c$, with $c\in K$. Let $\delta_1$ be as above. Consider the conditions
	\[
	\begin{array}{lllll}
		& {\bf A.} \hspace{0.3cm} &\z_3\notin K; \hspace{1.5cm}  & \\
		& {\bf B1.} \hspace{0.3cm} &\z_7+\z_7^{-1}\notin K(\z_3);  & {\bf C.} \hspace{0.3cm} &\sqrt[3]{\delta_1c}\notin K(\z_3,\z_7); \\
		& {\bf B2.} \hspace{0.3cm} & \z_7\notin K(\z_3,\z_7+\z_7^{-1});  \hspace{1.5cm} & {\bf D.} \hspace{0.3cm} &\sqrt{(\delta_1+1)c}\notin K(\z_3,\z_7).
	\end{array}
	\]

	The possible degrees of the extension $K_7/K$ are the following
	
	\begin{center}
		\begin{tabular}{|c|c|c|c|}
			\hline
			
			$d$ & {\em holding conditions}  &  $d$ & {\em holding conditions}  \\
			\hline
			{\em 72} & {\bf A}, {\bf B1}, {\bf B2}, {\bf C}, {\bf D} &  {\em 8} &  {\bf A}, {\bf B2}, {\bf D} \\
			\hline
		\begin{minipage}[t]{0.4cm}	\begin{center}\vskip 0.003cm		{\em 36} \vskip 0.003cm	\end{center}\end{minipage}&\begin{minipage}[t]{6cm}	\begin{center}\vskip 0.003cm {\bf B1}, {\bf C} \textrm{and two among } {\bf A}, {\bf B2}, {\bf D}  \vskip 0.003cm	\end{center}\end{minipage}
& \begin{minipage}[t]{0.4cm}	\begin{center}\vskip 0.003cm	 {\em 6}  \vskip 0.003cm	\end{center}\end{minipage} & \begin{minipage}[t]{6cm}\begin{center}\textrm{one between } {\bf B1}, {\bf C} \textrm{and} \\ \textrm{one among } {\bf A}, {\bf B2}, {\bf D}\end{center}\end{minipage}\\
			\hline
			{\em 24} & \textrm{one between } {\bf B1}, {\bf C} and {\bf A}, {\bf B2}, {\bf D}  & {\em 4} & \textrm{two among } {\bf A},  {\bf B2}, {\bf D}   \\ 
			\hline
			{\em 18} & {\bf B1}, {\bf C} \textrm{and one among } {\bf A}, {\bf B2}, {\bf D}   & {\em 3} & \textrm{one between } {\bf B1}, {\bf C}   \\
			\hline
	\begin{minipage}[t]{0.4cm}	\begin{center}\vskip 0.003cm			{\em 12}  \vskip 0.003cm	\end{center}\end{minipage} & \begin{minipage}[t]{6cm}\begin{center}\textrm{one between } {\bf B1}, {\bf C}  \textrm{and} \\  \textrm{two among } {\bf A}, {\bf B2}, {\bf D}\end{center}\end{minipage} & \begin{minipage}[t]{0.4cm}	\begin{center}\vskip 0.003cm	 {\em 2}  \vskip 0.003cm	\end{center}\end{minipage} &\begin{minipage}[t]{6cm}	\begin{center}\vskip 0.003cm	 \textrm{one among } {\bf A},  {\bf B2}, {\bf D}   \vskip 0.003cm	\end{center}\end{minipage} \\
			\hline
			{\em 9} & {\bf B1}, {\bf C} &  {\em 1} & \textrm{no conditions hold}   \\
			\hline
			\multicolumn{4}{c}{  }\\
			
			\multicolumn{4}{c}{{\em Table 2}}
			
		\end{tabular}
	\end{center}
	
\end{theorem}

\begin{proof}
	Consider the tower of extensions
	\[ K\subseteq K(\z_3)\subseteq K\left(\z_3,\z_7+\z_7^{-1}\right) \subseteq K(\z_3,\z_7) \subseteq K\left(\z_3,\z_7,\sqrt[3]{\delta_1c}\right)
	\subseteq K\left(\z_3,\z_7,\sqrt[3]{\delta_1c}, \sqrt{(\delta_1+1)c}\right).\]
	
\noindent 	We have that each of the degrees $\left[K(\z_3):K\left(\z_3,\z_7+\z_7^{-1}\right)\right]$ and $\left[K\left(\z_3,\z_7,\sqrt[3]{\delta_1c}\right):K(\z_3,\z_7)	\right]$ divides 3. In addition each of the degrees $[K(\z_3):K]$, $\left[K(\z_3,\z_7):K\left(\z_3,\z_7+\z_7^{-1}\right)\right]$ and $ \left[K_7: K\left(\z_3,\z_7,\sqrt[3]{\delta_1c}\right)\right]$ divides 2.
The conclusions follow immediately from $[K_7:K]$ being the product of the degrees of the intermediate extensions appearing in the tower.
\end{proof}

\noindent Notice that in this case too we have $[K_7:K]\leq 72< 366=|\GL_2(\Z/7\Z)|$ and the
Galois representation

\[ \rho_{\E_2,7}:{\Gal}(\ov{K}/K) \rightarrow  \GL_2(\Z/7\Z) \]

\noindent is not surjective, in accordance with $\E_2$ having complex multiplication.

\section{Galois groups ${\Gal}(K_7/K)$ for the curves of $\mcF_2$} \label{sub2}

Let $\E_2$ be a curve of the family $\mcF_2$. We are going to show all possible Galois groups $\Gal(K(\E_2[7])/K)$,
with respect to the degrees $d=[K_7:K]\leq 72$.

\begin{theorem} \label{galF2}
Let $K$ be a field with $\chara(K)\neq 2,3$ and let $\E_2$ be an elliptic curve with Weierstrass form
$y^2=x^3+c$, where $c\in K$. Then $\Gal(K_7/K)$ is isomorphic to a subgroup
of $G\simeq \Dic_3\rtimes \Z/6\Z$. In particular, if $[K_7:K]=72$, then $\Gal(K_7/K)\simeq \Dic_3\rtimes \Z/6\Z$.
\end{theorem}

\begin{proof}
Suppose that all the conditions in Theorem \ref{conditions2} hold, so that $[K_7:K]=72$. The image of $\Gal(\overline{K}/K)$ via the Galois representation $\rho_{\E_2,7}$ is a subgroup of $\GL_2(\Z/7\Z)$ isomorphic to $G=\Gal(K_7/K)$. As for the family $\mcF_1$, we denote by $G$ both $\Gal(K_7/K)$ and its image in $\GL_2(\Z/7\Z)$. Consider the tower of extensions in Figure \ref{fig:tower}. 
We denote by $H$ both $\Gal(K_7/K(\z_7))$ and its image in $\GL_2(\Z/7\Z)$.
The Galois group $\Gal(K(\z_7)/K)$ is then isomorphic to the quotient $G/H$. The group $H$ has order 12, because of $\Gal(K(\z_7)/K) \simeq \Z/6\Z$.  Since the action of $\sigma \in \GL_2(\Z/7\Z)$ on $\z_7$ is given by $\sigma (\z_7): = \z_7^{\det\left(\sigma\right)}$, then $\det\left(\sigma\right)=1$, for every $\sigma \in H$. Thus $H$ is indeed a subgroup of $\SL_2(\Z/7\Z)$. We are going to describe $H$ up to isomorphism. For every positive integer $n$, we denote by $\D_{2n}$ the dihedral group of order $2n$. Since we are assuming that all the conditions in Theorem \ref{conditions2} hold, then the complex multiplication $\phi_2$ and $-\Id$ are automorphisms of $H$.
The complex multiplication $\phi_2$ has order 3 and
 acts on the basis $\{P_1,\phi_2(P_1)\}$ as
	$$P_1 \xmapsto{\phi_2} \phi_2(P_1), \quad \phi_2(P_1) \xmapsto{\phi_2} \phi_2^2(P_1).$$

\noindent	Since $\phi_2^2(P_1)= -P_1 - \phi_2(P_1)$, then we can represent $\phi_2$ in $\GL_2(\Z/7\Z)$ as
	\[ \phi_2 = \left(\begin{array}{cc}
		0 & -1 \\
		1 & -1 \\
	\end{array}\right).	
	\]
	\noindent Then the inverse of $\phi_2$ is represented by the matrix
	\[ \phi_2^{-1}=\phi_2^2 = 
		\left(\begin{array}{cc}
			-1 & 1 \\
			-1 & 0 \\
		\end{array}\right).
	\]
	
\noindent The automorphism $-\Id$, swapping the ordinates	$P \overset{-\Id}{\longmapsto}-P$, for every $P\in \E_2[7]$,
	corresponds to the automorphism of $K_7/K$ that maps $\sqrt{(\delta_j+1)c}$ to $-\sqrt{(\delta_j+1)c}$, for all  $1\leq j\leq 8$.
Clearly $\phi_2$ and $-\Id$ commute, so $H$ has a subgroup isomorphic to $\Z/3\Z \times \Z/2\Z \simeq \Z/6\Z$.
We are going to show that $H$ is not abelian. Suppose that $H$ is abelian. Then it is isomorphic to either $\Z/3\Z\times \Z/4\Z$ or $\Z/3\Z \times \left(\Z/2\Z\right)^2$. Let $\sigma \in H$. Since $\sigma$ commutes with $\phi_2$, one gets that

$$\sigma=\left(\begin{array}{cc}
	\a & -\b\\
	\b & \a-\b\\
\end{array}\right),$$ 

\noindent for some $\a,\b\in \Z/7\Z$. We have already observed that $-\Id\in H$. 
If $H \simeq \Z/3\Z\times \Z/4\Z$, then there exists $\sigma \in H$ such that $\sigma^4 = \Id$ and $\sigma^2 = -\Id$, i.e.
\[
\sigma^2 =\left(\begin{array}{cc}
	\a^2-\b^2 & \b^2-2\a\b\\
	2\a\b-\b^2 & \a^2-2\a\b\\
\end{array}\right)
\equiv \left(\begin{array}{cc}
	-1 & 0 \\
	0 & -1 \\
\end{array}\right) (\modn 7).
\]

\noindent The congruence $\b^2-2\a\b\equiv 0 \hspace{0.1cm} (\modn 7)$ implies $\b = 0$ or $\b = 2\a$. If $\b=0$, then  $\a^2\equiv -1  \hspace{0.1cm} (\modn 7)$, which has no solutions. If $\b = 2\a$, then $-3\a^2 \equiv -1  \hspace{0.1cm} (\modn 7)$,
which has no solutions
as well. 
On the other hand, if $H \simeq \Z/3\Z \times \left(\Z/2\Z\right)^2$, then there exists $\sigma \in H$ such that $\sigma^2 = \Id$ and $\sigma \neq \pm\Id$. By
\[
\sigma^2 =\left(\begin{array}{cc}
	\a^2-\b^2 & \b^2-2\a\b\\
	2\a\b-\b^2 &\a^2-2\a\b\\
\end{array}\right)
\equiv \left(\begin{array}{cc}
	1 & 0 \\
	0 & 1 \\
\end{array}\right) (\modn 7),
\]

\noindent we get again $\b=0$ or $\b = 2\a$. If  $\b=0$, then $\sigma = \pm \Id$ and we have a contradiction. Suppose $\b = 2\a$. By $\det(\sigma)\equiv 1 \hspace{0.1cm} (\modn 7)$ we get $\alpha^2 \equiv -2\hspace{0.1cm} (\modn 7)$, which has no solutions. Therefore $H$ is not abelian, as claimed. In addition $H$ is a group of order 12, with  a subgroup isomorphic to $\Z/6\Z$. We have that
either $H \simeq \D_{12}$ or $H\simeq \Dic_{3}$, where $\Dic_3$ is the dicyclic group of order 12 (which is isomorphic to $\Z/3\Z \rtimes \Z/4\Z$). 
Suppose that $H \simeq \D_{12}$. We also have $H \simeq \D_6 \times \Z/2\Z$. Then $H$ is generated by $-\Id$,
$\phi_2$ and another automorphism $\tau$ of order $2$ such that $\phi_2\tau = \tau \phi_2^{-1}$
(i.e. $\langle \phi_2, \tau\rangle\simeq \D_6$ and $H=\langle -\Id\rangle \times \langle \phi_2, \tau\rangle\simeq \Z/2\Z\times \D_6\simeq \D_{12}$). The relation $\phi_2\tau = \tau \phi_2^{-1}$ implies that $\tau$ is represented by a matrix of the form
\[\tau= \left(\begin{array}{cc}
	\a & \b\\
	\a+\b & -\a\\
\end{array}\right),\]
\noindent for some $\a,\b \in \Z/7\Z$.
Since $\tau$ has order $2$, then

\[
\tau^2 =\left(\begin{array}{cc}
	\a^2+\b^2+\a\b & 0\\
	0 &\a^2+\b^2+\a\b\\
\end{array}\right)
\equiv \left(\begin{array}{cc}
	1 & 0 \\
	0 & 1 \\
\end{array}\right) (\modn 7),
\]

\noindent i.e.\ $\a^2+\b^2+\a\b \equiv 1 \hspace{0.1cm} (\modn 7)$. On the other hand, since $\det(\tau)\equiv 1 \hspace{0.1cm} (\modn 7)$, then $\a^2+\b^2+\a\b \equiv -1 \hspace{0.1cm} (\modn 7)$ and we have a contradiction.
Therefore the group $H$ is isomorphic to $\Dic_3\simeq \ZZ/3\ZZ\rtimes \ZZ/4\ZZ$ and it is generated by the complex multiplication $\phi_2$ of order 3 and an automorphism $\tau_2$ of order $4$, such that 

\[ H=\la \phi_2, \tau_2 \mid \phi_2^3 = \tau_2^4 = 1, \, \phi_2\tau_2 = \tau_2\phi_2^{-1} \ra. \]

\noindent In addition $\tau_2^2=-\Id$. Since $\GL_2(\Z/7\Z)\simeq \SL_2(7) \rtimes  \FF_7^*\simeq \SL_2(7) \rtimes \Z/6\Z$, then $G\simeq \Dic_3\rtimes \Z/6\Z$. For completeness we are going to show that this last semidirect product is not a direct product (as in the case of the family $\mcF_1$). %In doing this, we will get information on the possible representation of $\tau$ in $\GL_2(\Z/7\Z)$.
The group $G/H$ is generated by an automorphism $\psi_2$ of order 6 corresponding to the automorphism $\sigma_2$ of $\Gal(\Q(\z_7)/\Q)$ mapping $\z_7$ to $\z_7^{5}$. As stated above, we have

\[
\delta_1 \xmapsto{\sigma_2} \delta_2 \xmapsto{\sigma_2}  \delta_3 \xmapsto{\sigma_2}  \delta_4 \xmapsto{\sigma_2}  \delta_5 \xmapsto{\sigma_2}  \delta_6 \xmapsto{\sigma_2}  \delta_1.
\]

\noindent Then $\psi_2(P_1)$ is one of the points $\pm \phi_2^n(P_2)$, with $n\in \{0,1,2\}$ (where $\phi_2^0=\Id$). Since
$\phi_2(x_1)=\z_3x_1$, then $\psi_2(\phi_2(P_1))$ is one of the points $\pm \phi_2^{n+1}(P_2)=\pm \phi_2(\psi_2(P_1))$. 
Suppose that $\psi_2(P_1) =\alpha P_1 + \beta \phi_2(P_1)$, for some $\alpha, \beta\in \Z/7\Z$. Then
$\psi_2(\phi_2(P_1)) =\pm(\beta P_1 + (\beta-\alpha) \phi_2(P_1))$. Thus

$$\psi_2=\left(\begin{array}{cc}  \alpha & -\beta \\ \beta &-\beta+\alpha \end{array}\right) \quad\text{or}\quad \psi_2=\left(\begin{array}{cc}  \alpha & \beta \\ \beta & \beta-\alpha \end{array}\right). $$

\noindent Only in the first case $\psi_2$ and $\phi_2$ commute. By \cite[Chapter II, Theorem 2.3]{Sil2}, the extension  $K_7/K(\z_3)$ is abelian.\ If all the conditions in Theorem \ref{conditions2} hold, then the complex multiplication $\phi_2$ and $\psi_2$ are automorphisms of
$\Gal(K_7/K(\z_3))$. Therefore $\phi_2$ and $\psi_2$ commute and

$$\psi_2=\left(\begin{array}{cc}  \alpha & -\beta \\ \beta &-\beta+\alpha \end{array}\right).$$

\noindent Observe that $\psi_2^2$ maps $P_1$ to $\phi_2^j(P_3)$, for some $j\in\{0,1,2\}$ and $\phi_2(P_1)$ to $\phi_2^{j+1}(P_3)$. As noted in the proof of Theorem \ref{zeta3}, we have $x(P_3) = x(2P_1)$, i.e. $P_3 = 2P_1$ or $P_3=-2P_1$. We also have
 $x(\phi_2(P_3)) = x(2\phi_2(P_1))=\phi_2(x(2P_1))$. Therefore the automorphisms $\psi_2^2$ is equal to $\phi_2^j\omega$, where $\omega$ is represented by one of the following matrices

\[ \omega = \begin{pmatrix}
	2 & 0 \\
	0 & 2 \\
\end{pmatrix}
\quad \text{or} \quad
\omega = \begin{pmatrix}
	-2 & 0 \\
	0 & -2 \\
\end{pmatrix}.
\]

\noindent The second case is not possible, since $\psi_2^2$ would not have order 3.
Therefore $\psi_2^2=\phi_2^j\omega$, with

\[ \omega = \begin{pmatrix}
	2 & 0 \\
	0 & 2 \\
\end{pmatrix}. \]

\noindent Since $\psi_2^2\in G/H$ and $\phi_2^j\in H$, we may assume without loss of generality
that $\psi_2^2=\omega$, by eventually changing the representative of the class $\psi_2^2$ in $G/H$.
Observe that then $\psi_2^2$ commutes with every other automorphism of $G$. Furthermore, we have

\[
\psi_2^2 = \begin{pmatrix}
	\a^2-\b^2 & \b^2-2\a\b \\
	2\a\b-\b^2 & \a^2-2\a\b \\
\end{pmatrix} \equiv 
\begin{pmatrix}
	2 & 0 \\
	0 & 2 \\
\end{pmatrix} (\modn 7),
\]

\noindent Thus $\b(\b-2\a) = 0$, implying $\b = 0$ or $\b = 2\a$. If $\b = 0$, then $\det(\psi_2) = \a^2 \equiv 5 \hspace{0.1cm} (\modn 7)$, which has no solution (recall that $\psi_2(\z_7)=\z_7^5=\z_7^{\det(\psi_2)}$). So $\b = 2\a$ and $-3\a^2\equiv 2\hspace{0.1cm} (\modn 7)$, i.e. 
$\a \equiv 2 \hspace{0.1cm} (\modn 7)$ or $\a \equiv -2 \hspace{0.1cm} (\modn 7)$. Thus 

\[
\psi_2 = \begin{pmatrix}
	2 & 3 \\
	-3 & -2\\
\end{pmatrix} \quad \text{or} \quad
\psi_2 = \begin{pmatrix}
	-2 & -3 \\
	3 & 2 \\
\end{pmatrix}
= -\begin{pmatrix}
2 & 3 \\
-3 & -2\\
\end{pmatrix}.
\]

\noindent Again, by eventually change the representative of the class of $\psi_2$ in $G/H$, we may assume without
loss of generality that

$$\psi_2=\begin{pmatrix}
	-2 & -3 \\
	3 & 2\\
\end{pmatrix}.$$

\noindent We consider an automorphism $\rho\in H$ induced by the automorphism of $\Gal(K_7/K)$ mapping $\z_3$ to $\z_3^2$.
Thus $\rho$ maps $P_1$ to $\phi_2^j(P_4)$, for some $j\in \{0,1,2\}$ and we have that
there exists a power $\phi_2^s$ of $\phi_2$, with $s \in \{0,1,2\}$ such that $j+s \equiv 0 \mod 3$. We call $\tilde{\rho}$ the product $\phi_2^s\rho$ and we have that it maps $P_1$ to $P_4$. Since $\psi_2^3$ also maps $P_1$ to $P_4$ and

$$\psi_2^3=\begin{pmatrix}
	3 & 1 \\
	-1 & 4\\
\end{pmatrix},$$

\noindent  then we have

$$\tilde{\rho}=\left(\begin{array}{cc}
	3 & \alpha\\
	-1& \beta
\end{array}\right),$$

\noindent for some $\alpha,\beta\in \Z/7\Z$. Thus

$$\tilde{\rho}\psi_2-\psi_2\tilde{\rho}=\left(\begin{array}{cc}
	3\a-3 & 3\b+4\a-2\\
	3\b-5&-3\a+3
\end{array}\right).$$

\noindent Therefore $\tilde{\rho}$ and $\psi_2$ commute if and only if $\a\equiv 1\hspace{0.1cm} (\modn 7)$ and
$\b\equiv -3\hspace{0.1cm} (\modn 7)$. But then $\det(\tilde{\rho})=-1$ and we would have a contraddiction
with $\rho\in H$. Therefore $G\simeq \Dic_3\rtimes \Z/6\Z$ with
$G\neq \Dic_3\times \Z/6\Z$. For every $d=[K_7:K]<72$, we have that $G$ is
isomorphic to a proper subgroup of $\Dic_3\rtimes \Z/6\Z$.
\end{proof}

Observe that the situation for the Galois groups of the family $\mcF_2$ is similar to that of the Galois groups of the family $\mcF_1$; in fact for the curves in $\mcF_1$
have that $G\simeq \Dic_4\rtimes \Z/6\Z$, since the dicyclic group  $\Dic_4$ of order 16 is nothing but the
quaternion group $Q_{16}$.\par

\bigskip  We are going to describe the possible Galois groups $G=\Gal(K(\E_2[7])/K)$ when $d\leq 72$. By the proof of Theorem \ref{galF2}, we notice that $\psi_2^3=2\psi_2$ does not commute with $\tilde{\rho}$ too.
In addition we deduce that  $\tau_2$ and $\psi_2$ do not commute (otherwise we would get $G\simeq \Dic_3\times \Z/6\Z$) and that
$\tau_2$ and $\psi_2^3$ do not commute as well. Recall that by the mentioned \cite[Chapter II, Theorem 2.3]{Sil2}, we have that $G$
is abelian whenever {\bf A} does not hold.  Recall also that every nontrivial proper subgroup of $\Dic_3\simeq \Z/3\Z\rtimes \Z/4\Z$ is isomorphic to $\Z/m\Z$, with $m \in \{2,3,4,6\}$. 
In particular if {\bf D} does not hold, then $H$ is abelian.  If {\bf D} does not hold and $\phi_2\in H$, then we have that
every other automorphism of $H$ commutes with $\phi_2$ and it is then represented by a matrix of the form 

$$\begin{pmatrix}
	\a & -\b\\
	\b & \a-\b\\
\end{pmatrix},$$ 

\noindent for some $\a$ and $\b$ in $\Z/7\Z$. Every matrix of this type commutes with $\psi_2$ too.
Therefore, if {\bf D} does not hold, then $G$ is abelian as well.

\bigskip \emph{Galois groups $\Gal(K(\E_2[7])/K)$}

\begin{description}
	\item[$d=72$]
	\par If the degree $d$ of the extension $K_7/K$ is 72, then all the conditions hold. We have proved in Theorem \ref{galF2}
that	in this case $G\simeq \Dic_3\rtimes \Z/6\Z$.
	
	\item[$d=36$]
	If the degree $d$ of the extension $K_7/K$ is 36, then condition {\bf B1} and condition {\bf C}	 hold.
 
	\begin{description}
\item[-] If one between {\bf A} and  {\bf D} does not hold, then we have an abelian group. Since both condition {\bf B1} and condition {\bf B2} hold, then $G/H \simeq \ZZ/6\ZZ$ and thus $G\simeq \Z/3\Z\times \Z/2\Z \times \Z/6\Z\simeq (\Z/2\Z)^2\times (\Z/3\Z)^2$. 
	\item[-]  If {\bf B2} does not hold, then $G/H\simeq \Z/3\Z$ and $G=\langle \phi_2, \tau_2, \psi_2^2\rangle$. Since $\psi_2^2$ is represented by a diagonal matrix and commutes with every other automorphism, then $G\simeq \Dic_3\times \Z/3\Z$.
		\end{description}

	\item[$d=24$] 
	Only one between condition {\bf B1} and condition {\bf C} hold and all the other conditions hold.
	
	\begin{description}
\item[-] If {\bf C} does not hold, then	 $G=\langle  \tau_2, \psi_2 \rangle\simeq \Z/4\Z \rtimes \Z/6\Z\simeq \D_8 \times \Z/3\Z$.
	\item[-] If {\bf B1} does not hold, then $G/H\simeq \Z/2\Z=\langle \psi_2^3\rangle$. Thus $G\simeq \Dic_3 \rtimes \Z/2\Z$.
	\end{description}

	\item[$d=18$]
	Conditions {\bf B1} and {\bf C} hold and the automorphism $\phi_2$ has order 3. Since only one among the
other conditions holds, then
at least one between {\bf A} or {\bf D} does not hold and $G$ is abelian.
	\begin{description}
\item[-] If either {\bf A} or {\bf D}  holds, then $G \simeq (\Z/3\Z)^2 \times \Z/2\Z$.
	\item[-]  If both {\bf A} and {\bf D} do not hold, then {\bf B2} holds. We have $G/H \simeq \langle \psi_2\rangle \simeq \Z/6\Z$ and $H\simeq\langle \phi_2\rangle\simeq \Z/3\Z$. Since $\phi_2$ and $\psi_2$ commutes, then $G \simeq (\Z/3\Z)^2 \times \Z/2\Z$ as well.
\end{description}

	\item[$d=12$]
	Only one between conditions {\bf B1} and {\bf C} holds and two among the other conditions hold.
	\begin{description}
\item[-] If  both {\bf B1} and {\bf B2} do not hold, then the extension $G/H$ is trivial and
	$G=H\simeq\Dic_3 \simeq \Z/3\Z \rtimes \Z/4\Z$.  
	\item[-] If {\bf B1} does not hold and {\bf B2} holds, then  $H \simeq\Z/2\Z$ and $G/H\simeq \Z/6\Z$. One between {\bf A} and  {\bf D} does not hold. Then the extension $K_7/K$ is abelian with Galois group $G \simeq \Z/3\Z \times (\Z/2\Z)^2$.
\item[-] If {\bf B1} holds and {\bf B2} does not hold, then $G/H$ is generated by $\psi_2^2$ and $H = \langle \tau_2 \rangle \simeq \Z/4\Z$. Since $\psi_2^2$ commutes with $\tau_2$, then the Galois group $G$ is isomorphic to $\Z/4\Z \times \Z/3\Z$.
%	\par If {\bf B1} does not hold and {\bf B2} holds, then $H =\Z/6\Z$ and $G/H=\langle \psi_2^2\rangle\simeq \Z/2\Z$. Then $K_7/K$ is an abelian extension with Galois group $G \simeq \Z/3\Z \times (\Z/2\Z)^2$.
	\item[-] If both {\bf B1} and {\bf B2} hold, then $G/H\simeq \Z/6\Z$ and $H \simeq \Z/2\Z$. We have that  
one between  {\bf A} and {\bf D} does not hold. Hence $K_7/K$ is an abelian extension with Galois group $G \simeq \Z/3\Z \times (\Z/2\Z)^2$.
	\end{description}

	\item[$d=9$]
	The only holding conditions are {\bf B1} and {\bf C}. Then $H=\langle \phi_2\rangle \simeq \Z/3\Z$ and $G/H\simeq \Z/3\Z$. We have $G\simeq \ZZ/3\ZZ\times \Z/3\Z$.

	\item[$d=8$]
	If  the degree $d$ of the extension $K_7/K$ is 8, then all the conditions hold but {\bf B1} and {\bf C}. Thus $H=\langle \tau_2\rangle\simeq \Z/4\Z$ and $G/H=\langle \psi_2^3\rangle \simeq \Z/2\Z$. We have observed that
$\tau_2$ and $\psi_2^3$ do not commute. Then $G\simeq \Z/4\Z\rtimes \Z/2\Z\simeq \D_8$.

	\item[$d=6$]
	If the degree $d$ of the extension $K_7/K$ is 6, then either {\bf B1} or {\bf C} holds and one among the other condition holds. In all cases the group $G$ is isomorphic to $\Z/6\Z \simeq \ZZ/3\ZZ\times \Z/2\Z$.
	
	\item[$d=4$]
	If  the degree $d$ of the extension $K_7/K$ is 4, then both {\bf B1} and {\bf C} do not hold. 
	\begin{description}
\item[-]  If {\bf B2} does not hold, then $G/H$ is trivial and $G = H = \la \tau_2 \ra \simeq \ZZ/4\ZZ$.
	\item[-] If {\bf B2} holds, then $G/H\simeq \Z/2\Z$ and $G$ is isomorphic to the Klein group  $\ZZ/2\ZZ \times \ZZ/2\ZZ$. 
		\end{description}

	\item[$d\leq 3$]
	If the degree $d$ of the extension $K_7/K$ is 3 or 2 or 1, obviously the Galois group is respectively $\Z/3\Z$,  $\Z/2\Z$ or $\{{\rm Id}\}$.

\end{description}

\section{Some applications} \label{lastsub}

As mentioned in Section \ref{sec1}, we are going to describe some applications of the results produced in the
previous sections.

\subsection{A minimal bound for the local-global divisibility by $7$} \label{loc-glob}

The first application concerns the following local-global question that was stated in \cite{DZ1} by R. Dvornicich and 
U. Zannier as a generalization of a particular case of the famous Hasse principle on quadratic forms (for further details one can see  \cite{DZ3}, \cite{Cre2}, \cite{GR2}, \cite{PRV2} and \cite{Pal_2019} among others;
Dvornicich and the corresponding author also produced a survey \cite{DP} about this topic).

\begin{Problem}[Dvornicich, Zannier, 2001] \label{prob1}
	Let $K$ be a number field, $M_K$ the set of the places $v$ of $K$ and $K_v$ the completion of $K$ at $v$.
	Let $\G$ be a commutative algebraic group defined over $K$. Fix a positive integer $m$ 
	and assume that there exists a $K$-rational point $P$ in $\G$, such that  $P=mD_v$, for some $D_v\in {\mathcal{G}}(K_v)$,
	for all but finitely many  $v\in M_K$. Does there exist $D\in \G(K)$ such that $P=mD$?
\end{Problem}

\noindent  We have stated the question in its original form, for all commutative algebraic groups, but from here on out
we will confine the discussion to elliptic curves $\E$ over $K$. It is a common method 
in local-global questions to translate the problem in a cohomological question. 
Dvornicich and  Zannier stated the following definition of a subgroup of
$H^1(G,\E[m])$ which encodes the hypotheses of the
problem and whose triviality assures the validity of the local-global divisibility by $m$ 
in $\E$ over $K$ \cite[Proposition 2.1]{DZ1}:

\begin{equation} \label{h1loc}
	H^1_{\loc}(G,\E[m]):=\bigcap_{v\in \Sigma} (\ker  H^1(G,\E[m])\xrightarrow{\makebox[1cm]{{\small $res_v$}}} H^1(G_v,\E[m])),
\end{equation}

\noindent where $\Sigma$ is the set of places of $K$ unramified in $K(\E[m])$ and $res_v$ is the usual restriction map.
  
\par\bigskip\noindent The group $H^1_{\loc}(G,\E[m])$  is called \emph{first local cohomology group} and 
gives an obstruction to the validity of this Hasse principle for divisibility of points by $m$ in $\E$
over a finite extensions of $K$ linearly disjoint from
$K(\E[m])$ \cite[Theorem 3]{DZ3}.
Since every $v\in \Sigma$ is unramified in $K(\E[m])$,  then $G_v$ is a cyclic subgroup of $G$, for all $v\in \Sigma$.
By the Chebotarev Density Theorem, the local Galois group $G_v$ varies over \emph{all} cyclic subgroups of
$G$ as $v$ varies in $\Sigma$. 
Observe that indeed it suffices to take

\begin{equation} \label{def_S}
	H^1_{\loc}(G,\E[m])=\bigcap_{v\in S} (\ker  H^1(G,\E[m])\xrightarrow{\makebox[1cm]{{\small $res_v$}}} H^1(G_v,\E[m])),
\end{equation}

\noindent with $S$ a subset of $\Sigma$ such that $G_v$ varies
over all cyclic subgroups of $G$ as $v$ varies in $S$. Observe that in particular we can choose a finite set $S$ (on the contrary $\Sigma$ is not finite).
In \cite{DZ1} the authors showed that the local-global divisibility by a prime number $p$ holds in $\E$ over $K$
(this was also proved in \cite[Theorem 1]{Won} and a very similar statement was proved in \cite[Lemma 6.1 and its corollary]{Cas4} and \cite[Theorem 8.1]{Cas3}). In particular,  the local-global divisibility by $7$ holds in $\E$ over $K$. 
Thus, if we are able to find such
a set $S$ and prove that the local divisibility by $7$ holds for $P\in\E(K)$, for all $v\in S$, then
we get that $P$ is globally divisible by $7$, i.e.\ that $P$ has a $K$-rational $7$-divisor. So it suffices to have the local divisibility by $7$ for a finite number of suitable places to get the global divisibility by $7$. 
In \cite{DP2} R. Dvornicich and the corresponding author produced an explicit effective version of the
hypotheses of Problem \ref{prob1} in all elliptic curves over number fields, by producing
an explicit finite set $S$, for every positive integer $m$ and every elliptic curve $\E$. Such an effective version
is given by an upper bound $B(m,\E)$ (depending on $m$ and $\E$) to the places of $K$ unramified in $K_m$,
such that the validity of the local divisibility for all places less than $B(m,\E)$ assures the global divisibility (in the cases when the Hasse principle for
divisibility of points holds in $\E$ over $K$). With such a bound it is not necessary to take into account
the distinctness of the Galois groups $G_v$ in testing the local divisibility, 
since it is already assured by the density of places $v$ that are considered.\ However, for this reason the cardinality of the
set $S$ produced in \cite{DP2} is not as minimal as possible.\ It is indeed a very hard problem to obtain an analogue result with an explicit set $S$ of minimal cardinality (i.e.\ with the assumption that the local Galois groups
$G_v$, corresponding to the places in $S$, are pairwise distinct), for all positive integers $m$. 
% In particular we can choose a minimal set $S$ such that $G_v$ varies
%over all cyclic subgroups of $G$ as $v$ varies in $S$ and, moreover, $G_v$ and $G_w$ are
%pairwise distinct cyclic subgroups of $G$, for all $v, w\in S$, with $v\neq w$. 
It is also a difficult problem just to find the minimal possible cardinality for $S$ for every $m$.
In view of the results achieved for the Galois groups ${\Gal}(K_7/K)$ for the elliptic curves of the families 
$\mcF_1$ and $\mcF_2$, we give an answer to this last question when $m=7$ for the curves of these families
(in \cite{Pal_2018} an answer was given when $m=5$ for the curves of the same families).
For these curves we produce an upper bound to the cardinality of $S$ 
which is surprisingly small and it is as minimal as possible
when the degree $[K_7:K]$
is maximum (i.e. $[K_7:K]=96$ for the curves in $\mcF_1$ and $[K_7:K]=72$ for the curves in $\mcF_2$).
With the description of the Galois groups given in Section \ref{gal1} and Section \ref{sub2} and with
the description of the cyclic subgroups of $G$ given in the proofs of  the following Theorem \ref{appl_1} and Theorem \ref{appl_2},
one can easily deduce the minimal cardinality for $S$, for every $\E_1\in  \mcF_1$ and
$\E_2\in  \mcF_2$.

\begin{theorem} \label{appl_1}
	Let $\E_1$ be an elliptic curve defined over a number field $K$, with Weierstrass 
	equation $y^2=x^2+bx$, for some $b\in K$.
	There exist sets $S\subseteq M_K$ of cardinality $s\leq 18$ such that if $P=7D_v$, with $D_v\in \E_1(K_v)$, for all $v\in S$,
	then $P=7D$, for some $D\in \E_1(K)$. In particular, if $[K_7:K]=96$, then $s=18$.
\end{theorem}

\begin{proof}
	Let $s$ be the number of distinct cyclic subgroups of $G$. As stated above, the set $S$ can be chosen as a subset of $M_k$ with cardinality $s$, such that $G_v$ varies over all cyclic subgroups of $G$, as $v$ varies in $S$, and $G_v$ and $G_w$ are
	pairwise distinct cyclic subgroups of $G$, for all $v, w\in S$, with $v\neq w$. It suffices to show that $s\leq 18$, i.e. that $G$ has at most	18 cyclic subgroups.
	We have proved in  Section \ref{gal1}, that for every $\E_1\in \mcF_1$, the Galois group $G$ is
	isomorphic to a subgroup of $Q_{16} \rtimes \Z/6\Z$. We keep the notation used in Section \ref{gal1} for the generators of $Q_{16}$ and $\Z/6\Z$, i.e. $Q_{16}=\langle \phi_1, \varphi_1 | \phi_1^2=\varphi_1^{4}=-\Id, \phi_1\varphi_1=\varphi_1^{-1}\phi_1\rangle$ and $\Z/6\Z=\langle \psi_1 \rangle$. 
	The group $Q_{16}$ has 7 nontrivial cyclic subgroups: $\langle \varphi_1\rangle \simeq  \Z/8\Z$, $\langle -\Id \rangle= \langle\phi_1^2 \rangle= \langle\varphi_1^4\rangle\simeq \Z/2\Z$ and the 5 cyclic subgroups of order $4$ generated respectively by $\phi_1$, $\varphi_1^2$, $\phi_1\varphi_1$, $\phi_1\varphi_1^2$ and $\phi_1\varphi_1^3$. We also have the nontrivial cyclic subgroups of $\Z/6\Z$, i.e  $\langle\psi_1^3\rangle\simeq \Z/2\Z$, $\langle\psi_1^2\rangle\simeq \Z/3\Z$ and $\langle\psi_1\rangle\simeq \Z/6\Z$ itself.
All of these groups are cyclic subgroups of $G$.	In addition we have the group $\langle \varphi_1,\psi_1^2\rangle\simeq \ZZ/8\ZZ\times \ZZ/3\ZZ\simeq \Z/24\Z$, 5 copies of $\Z/12\Z\simeq\ZZ/4\ZZ\times \ZZ/3\ZZ$ given by the direct products of the 5 subgroups of order 4 of $Q_{16}$ with $\langle\psi_1^2\rangle$, the subgroup $\langle -\Id,\psi_1^2\rangle\simeq \ZZ/2\ZZ\times \ZZ/3\ZZ$ and
the trivial group $\langle \Id\rangle$. Therefore $Q_{16} \rtimes \ZZ/6\ZZ$ contains 18 cyclic subgroups and every subgroup $G$ of $ Q_{16} \rtimes \ZZ/6\ZZ$ has at most 18 cyclic subgroups. Thus $s\leq 18$. In particular, if $[K_7:K]=96$, then $G$ has exactly 18 cyclic subgroups and in this case $s=18$ is sharp (in fact, if $s< 18$, then the hypotheses of Problem \ref{prob1} are not satisfied).
\end{proof}

\begin{theorem} \label{appl_2}
	Let $\E_2$ be an elliptic curve defined over a number field $K$, with Weierstrass equation 
	$y^2=x^2+c$, for some $c\in K$.
	There exist sets $S\subseteq M_K$ of cardinality $s\leq 15$ such that if $P=7D_v$, with $D_v\in \E(K_v)$, for all $v\in S$,
	then $P=7D$, for some $D\in \E(K)$. In particular, if $[K_7:K]=72$, then $s=15$.
\end{theorem}

\begin{proof}
	Let $s$ be the number of distinct cyclic subgroups of $G$. By the discussion concerning the minimal possible cardinality of
	the set $S$ in equation \eqref{def_S},
	we can choose $S$ containing exactly $s$ places $v$, such that $G_v$ varies over all cyclic subgroups of $G$ as $v$ varies in $S$ and
	$G_v$ and $G_w$ are
	pairwise distinct cyclic subgroups of $G$, for all $v, w\in S$, with $v\neq w$. We have just to show that $G$ has at most 15 cyclic subgroups.
	As proved in  Section \ref{sub2},  for every $\E_2\in \mcF_2$, the Galois group $G$ is
	isomorphic to a subgroup of $\Dic_3\rtimes \Z/6\Z$. In the notation of Section \ref{sub2}, a presentation of the group $\Dic_3 \simeq \Z/3\Z\rtimes\Z/4\Z$ is $\langle \phi_2, \tau_2| \phi_2^3=\tau_2^4=\Id, \tau_2\phi_2=\phi_2^{-1}\tau_2\rangle$. We have $6$ nontrivial cyclic subgroups of $\Dic_3$: $\langle -\Id\rangle$, $\langle \phi_2\rangle$, $\langle \tau_2\rangle$,
$\langle \tau_2\phi_2\rangle=\langle \tau_2^3\phi_2\rangle \simeq \Z/4\Z$, $\langle \tau_2\phi_2^2\rangle=\langle \tau_2^3\phi_2^2\rangle \simeq \Z/4\Z$, $\langle -\phi_2\rangle=\langle -\phi_2^2\rangle \simeq \Z/6\Z$.
We have $3$ nontrivial cyclic subgroups of $\langle \psi_2\rangle \simeq \ZZ/6\ZZ$, i.e. $\langle \psi_2\rangle$,
$\langle \psi_2^2\rangle$, $\langle \psi_2^3\rangle$.
	In addition we have $2$ other cyclic subgroups of $\Dic_3 \rtimes \Z/6\Z$ isomorphic to $\Z/3\times \Z/2\Z$, i.e.\
 $\langle \phi_2, \psi_2^3\rangle$, $\langle -\Id, \psi_2^2\rangle$ and 
$3$ subgroups isomorphic to $\Z/12\Z\simeq \Z/4\Z\times \Z/3\Z$, i.e.\ $\langle \tau_2,\psi_2^2\rangle$,
$\langle \tau_2\phi_2, \psi_2^2\rangle$, $\langle \tau_2\phi_2^2, \psi_2^2\rangle$.
Finally we have the trivial subgroup $\langle \Id\rangle$.
	Thus $\Dic_3 \rtimes \Z/6\Z$ has 15 cyclic subgroups and $s\leq 15$. In particular, if $[K_7:K]=72$, then $G$ has exactly 15 cyclic subgroups and in this case the bound $s=15$ is sharp. 
\end{proof}

\subsection{Remarks on modular curves} 

We are going to deduce some information about CM points on modular curves by the results produced about the fields $K_7$. 

\subsubsection{On CM points of modular curves} \label{mod_sec}

Let $\mathcal{H}$ be the complex
upper half plane $\{z\in \mathbb{C}\,:\,Im\,z > 0\}$. The group  $\SL_2(\Z)$ acts
on $\mathcal{H}$ via the M\"obius trasformations

\[ \left( \begin{array}{cc} a & b \\ c & d \end{array} \right)z= \dfrac{az+b}{cz+d} \ . \]

\noindent By $\Gamma$ we denote a congruence group, i.e.\ a subgroup of $\SL_2(\Z)$ containing the  {\it principal congruence group of level $m$} 

\[ \Gamma(m) = \left\{ A \in \SL_2(\Z)\,\mid \,
A\equiv \left( \begin{array}{cc} 1 & 0 \\ 0 & 1 \end{array} \right) \pmod m \right\}, \ \]

\noindent  for some positive integer $m$.  When $m$ is minimal,
the congruence group is said to be {\it of level $m$}.  Important congruence groups of level $m$ are

\[ \Gamma_0(m) = \left\{ A \in \SL_2(\Z)\,\mid \,
A\equiv \left( \begin{array}{cc} * & * \\ 0 & * \end{array} \right) \pmod m \right\} \ ,\]
and
\[ \Gamma_1(m) = \left\{ A\in \SL_2(\Z)\,\mid \,
A\equiv \left( \begin{array}{cc} 1 & * \\ 0 & 1 \end{array} \right) \pmod m \right\} \ .\]

\noindent The quotient  $\mathcal{H}/\Gamma$ of $\mathcal{H}$ 
by the action of $\Gamma$, with the analytic structure induced by $\mathcal{H}$, is a Riemann surface, 
that is denoted by $Y_\Gamma$.  The modular curve $X_\Gamma\,$, associated to $\Gamma$, is the
compactification of $Y_\Gamma$  by the addition of a finite number of cusps, i.e. the
rational points corresponding to the orbits of $\mathbb{P}^1(\QQ)$ under $\Gamma$.\ The modular curves associated to the groups $\Gamma(m)$, $\Gamma_0(m)$ and $\Gamma_1(m)$
are denoted respectively by $X(m)$, $X_0(m)$ and $X_1(m)$.  
They are moduli spaces of families of elliptic curves with an extra structure of level $m$ as
follows (for further details see for example \cite{KM}, \cite{Kn} and \cite{Shi}):

	\begin{description}

		\item[ i)] non cuspidal points in $X_0(m)$ correspond to pairs $(\E,C_m)$, where $\E$ is an elliptic curve
		(defined over $\mathbb{C}$) and $C_m$ is a cyclic subgroup of $\E[m]$ of order $m$;
		
		\item[ ii)]  non cuspidal points in $X_1(m)$ correspond to pairs $(\E,P)$, where $\E$ is an elliptic curve
		(defined over $\mathbb{C}$) and $P$ is a point of order $m$;
		
		\item[ iii)] non cuspidal points in $X(m)$ correspond to triples $(\E,P,Q)$, where $\E$ is an elliptic curve
		(defined over $\mathbb{C}$) and $P$, $Q$ are points of order $m$ generating $\E[m]$.
		
	\end{description}

\medskip\noindent A \emph{CM point} on a modular curve is a point which corresponds to an elliptic curve with
complex multiplication. For every modular curve $X$, we denote by
$X(K)_{CM}$ the set of its $K$-rational CM points.

We can deduce the following facts from what showed in the previous sections 
(see in particular Theorem \ref{zeta3}).

\begin{proposition}  \label{mod_1}
	Let $K$ be a number field. Let  $\delta_j$ and $P_j$ be as in Section \ref{subgen1}, 
	for $1\leq j\leq 8$, and let $\E_{2,j,\gamma}:y^2 = x^3 + c_{j,\gamma}$, with $c_{j,\gamma}: = \delta_j^2(\delta_j+1)^3\gamma^6$, for some
	$\gamma \in \QQ$. If $\QQ(\z_3,\z_7)\subseteq K$, then 

\begin{description}
\item[i.] the pairs $(\E_{2,j,\gamma}, P_j)$, $(\E_{2,j,\gamma}, \langle P_j \rangle)$, with $1\leq j \leq 8$,
		define  $K$-rational CM points on $X_1(7)$ and respectively on $X_0(7)$.
	\item[ii.]	the triples $(\E_{2,j,\gamma}, P_j,\phi_2(P_j))$, with $1\leq j \leq 8$, define $K$-rational CM points on $X(7)$;
\item[iii.] in particular $X_0(7)(K)_{CM}\neq \emptyset$, $X_1(7)(K)_{CM}\neq \emptyset$ and $X(7)(K)_{CM}\neq \emptyset$. 
\end{description}
\end{proposition}

\begin{proof}
	If $1 \leq j \leq 6$, then
		$c_{j,\gamma}, \sqrt[3]{\delta_j c_{j,\gamma}}, \sqrt{\delta_j c_{j,\gamma}} \in \QQ(\z_3,\z_7)$. Since $\QQ(\z_3,\z_7)\subseteq K$, then the pairs $(\E_{2,j,\gamma}, P_j)$ and $(\E_{2,j,\gamma}, \langle P_j \rangle)$, for 
		$1 \leq j \leq 6$, define  $K$-rational  CM points of $X_1(7)$ and respectively of $X_0(7)$. Furthermore, the triples $(\E_{2,j,\gamma}, P_j,\phi_2(P_j))$ define $K$-rational CM points on $X(7)$. 
\par If  $j \in \{7, 8\}$, then 
		$c_{j,\gamma} \in \QQ(\z_3)$. We also have that $\sqrt[3]{\delta_j c_{j,\gamma}} = \delta_j(\delta_j+1)\gamma^2 \in \QQ(\z_3)$ and $\sqrt{\delta_j c_{j,\gamma}} = \delta_j(\delta_j+1)^2\gamma^3\in \QQ(\z_3)$. Owing to
		$\QQ(\z_3,\z_7)\subseteq K$, then the pairs $(\E_{2,j,\gamma}, P_j)$ and $(\E_{2,j,\gamma}, \langle P_j \rangle)$ define  $K$-rational  CM points of $X_1(7)$ and respectively of $X_0(7)$. Furthermore, the triples $(\E_{2,j,\gamma}, P_j,\phi_2(P_j))$ define $K$-rational CM points on $X(7)$.
\end{proof}

\noindent Moreover, from the results proved in Section \ref{sec_gen_1} and Section \ref{subgen1}, we can
immediately deduce the following propositions. 

\bigskip\begin{proposition} \label{mod_3}
	Let $K$ be an extension of $\QQ(i,\z_7)$. Let $\E_1\in \mcF_1$ and let $P\in \E_1[7]$ such that $\{P, \phi_1(P)\}$ is a generating set of $\E_1[7]$.
	Then  
	
	\begin{description}
		
		\item[i)] the pair $(\E_1, \langle P\rangle)$ defines a non-cuspidal $K$-rational  CM point of $X_0(7)$, if and only if $y(P)\in K;$
		\item[ii)] the pair $(\E_1, P)$ defines a non-cuspidal $K$-rational CM point of $X_1(7)$, if and only if $y(P)\in K$;
		\item[iii)] the triple $(\E_1,P,\phi_1(P))$ defines a non-cuspidal $K$-rational  CM point of $X(7)$, if and only if $y(P)\in K.$
	\end{description}
\end{proposition}

\bigskip\begin{proposition} \label{mod_2}
	Let $K$ be an extension of $\QQ(\z_3,\z_7)$. Let $\E_2\in \mcF_2$ and let $P\in \E_2[7]$ such that $\{P, \phi_2(P)\}$ is a generating set of $\E_2[7]$.
	Then  
	
	\begin{description}
		
		\item[i)] the pair $(\E_2, \langle P\rangle)$ defines a non-cuspidal $K$-rational  CM point of $X_0(7)$, if and only if $y(P)\in K;$
		\item[ii)] the pair $(\E_2, P)$ defines a non-cuspidal $K$-rational  CM point of $X_1(7)$, if and only if $y(P)\in K$;
		\item[iii)] the triple $(\E_2,P,\phi_2(P))$ defines a non-cuspidal $K$-rational  CM point of $X(7)$, if and only if $y(P)\in K.$
	\end{description}
	
\end{proposition}

\vspace{0.5cm} \indent % {\it
%A\,c\,k\,n\,o\,w\,l\,e\,d\,g\,m\,e\,n\,t\,s.\;}

%%%----------------------------------------------%%%
%%%                BIBLIOGRAPHY                  %%%
%%%  REFERENCES SHOULD BE LISTED ALPHABETICALLY  %%%
%%%----------------------------------------------%%%
\bigskip
\begin{center}

\end{center}

%-----------------------------------------------
% Authors (name, address, e-mail)
%-----------------------------------------------
\bigskip
\bigskip

\begin{minipage}[t]{10cm}
	\begin{flushleft}
		\small{
			\textsc{Jessica Alessandr\`i}
			\\* University of L’Aquila,
			\\* Via Vetoio, Coppito 1
			\\* Coppito (AQ), 67100, Italy
			\\*e-mail: jessica.alessandri@graduate.univaq.it
			
		}
	\end{flushleft}
\end{minipage}

\bigskip

\begin{minipage}[t]{10cm}
	\begin{flushleft}
		\small{
			\textsc{Laura Paladino}
			\\*University of Calabria,
			\\* Ponte Bucci, Cubo 30B 
			\\* Rende (CS), 87036, Italy
			\\*e-mail: laura.paladino@unical.it
			
		}
	\end{flushleft}
\end{minipage}

%---------------------------------------------------- THE END
\end{document}